\documentclass{amsart}

\usepackage{amssymb}
\theoremstyle{plain}
\newtheorem{theorem}{Theorem}[section]
\newtheorem{lemma}[theorem]{Lemma}

\newtheorem{prop}[theorem]{Proposition}

\theoremstyle{definition}

\theoremstyle{remark}

\begin{document}

\title[Sobolev orthogonal polynomials on the unit ball]
{Sobolev orthogonal polynomials on the unit ball via outward normal derivatives}

\author[A. M. Delgado]{Antonia M. Delgado}
\address[A. M. Delgado]{{Departamento de Matem\'atica Aplicada
and Instituto de Matem\'aticas (IEMath-GR) \\
Universidad de Granada\\
18071 Granada, Spain}}\email{amdelgado@ugr.es}

\author[L. Fern\'andez]{Lidia Fern\'andez}
\address[L. Fern\'andez]{{Departamento de Matem\'atica Aplicada
and Instituto de Matem\'aticas (IEMath-GR) \\
Universidad de Granada\\
18071 Granada, Spain}}
\email[Corresponding author]{lidiafr@ugr.es}

\author[D. Lubinsky]{Doron Lubinsky}
\address[D. Lubinsky]{{School of Mathematics\\ 
Georgia Institute of Technology\\
Atlanta, GA 30332, USA}}\email{lubinsky@math.gatech.edu}

\author[T. E. P\'erez]{Teresa E. P\'erez}
\address[T. E. P\'erez]{{Departamento de Matem\'atica Aplicada
and Instituto de Matem\'aticas (IEMath-GR) \\
Universidad de Granada\\
18071 Granada, Spain}}\email{tperez@ugr.es}

\author[M. A. Pi\~{n}ar]{Miguel A. Pi\~{n}ar}
\address[M. A. Pi\~{n}ar]{{Departamento de Matem\'atica Aplicada
and Instituto de Matem\'aticas (IEMath-GR) \\
Universidad de Granada\\
18071 Granada, Spain}}\email{mpinar@ugr.es}

\begin{abstract}
We analyse a family of mutually
orthogonal polynomials on the unit ball with respect to an inner
product which involves the outward normal derivatives on the sphere.
Using their representation in terms of spherical
harmonics, algebraic and analytic properties will be deduced. First,
we deduce explicit connection formulas relating classical multivariate
ball polynomials and our family of Sobolev
orthogonal polynomials. Then explicit representations for the norms and
the kernels will be obtained. Finally, the asymptotic
behaviour of the corresponding Christoffel functions is studied.
\end{abstract}



\maketitle

\section{Introduction}
\setcounter{equation}{0}

The term \textit{Sobolev orthogonal polynomials} usually refers to a family of polynomials which are
orthogonal with respect to an inner product which simultaneously involves functions and their derivatives.
In the one variable case this kind of orthogonality has been studied during the last 25 years, and it
constitutes the main subject of a vast literature (see \cite{MarcellanXu2015} and the references therein).

Sobolev orthogonal polynomials in several variables have a considerably shorter history. There are
very few references on the subject and most of them deal with Sobolev orthogonality on the unit
ball $\mathbb{B}^d$ of $\mathbb{R}^d$. Usually, the inner product considered is some modification
of the classical inner product on the ball
$$
  \langle f, g\rangle_\mu = \frac{1}{\omega_\mu}\int_{\mathbb{B}^d} f(x) g(x) W_\mu(x) dx,
$$
where $W_\mu(x) = (1-\|x\|^2)^\mu$ on $\mathbb{B}^d$, $\mu > -1$, and
$\omega_\mu$ is a normalizing constant such that $\langle 1 ,1 \rangle_\mu = 1$.

One of the first works on this subject was a paper by Y. Xu  \cite{Xu2008}, where the inner product
$$
\langle f,g \rangle_{I} = {\frac{\lambda}{\sigma_d}} \int_{\mathbb{B}^d} \nabla f(x) \cdot \nabla g(x) dx
+ {\frac{1}{\sigma_d}} \int_{\mathbb{S}^{d-1} } f(\xi) g(\xi) d\sigma(\xi), \quad \lambda  >0,
$$
was considered. Here, $d\sigma$ denotes the surface measure on the sphere $\mathbb{S}^{d-1}$ and
$\sigma_d$ denotes the surface area. In the same article, the author studied another inner product
where the second term on the right hand side was replaced by ${f(0)g(0)}$. In both cases,
the central symmetry of the inner products plays an essential role and using
spherical polar coordinates a mutually orthogonal polynomial basis is constructed. The polynomials in
this basis are expressed in terms of Jacobi polynomials and spherical harmonics mimicking the
standard construction of the classical ball polynomials.

\bigskip

In the present paper, we study orthogonal polynomials with respect to the Sobolev inner product
$$
 \langle f ,g \rangle_{\mu}^S
     =  \frac{1}{\omega_\mu} \int_{\mathbb{B}^d} f(x) g(x)  W_\mu(x)  dx +
    \frac{\lambda}{\sigma_d}  \int_{\mathbb{S}^{d-1}} \frac{\partial f}{\partial {\bf n}}(\xi)
    \frac{\partial g}{\partial {\bf n}}(\xi)d\sigma(\xi),
$$
where $\lambda>0$ and $\frac{\partial }{\partial {\bf n}}$ stands for the outward normal derivative operator.

\bigskip

Using again spherical polar coordinates, we shall construct a sequence of mutually orthogonal polynomials
with respect to $ \langle \cdot ,\cdot \rangle_{\mu}^S$, which depends
on a family of Sobolev orthogonal polynomials of one variable. The latter are usually called a
\textit{non--diagonal Jacobi Sobolev--type} family of orthogonal polynomials and can be expressed
in terms of Jacobi polynomials (see \cite{DuenasGarza2013}).

Standard techniques provide us explicit connection formulas relating classical multivariate
ball polynomials and our family of Sobolev orthogonal polynomials. The explicit representations
for the norms and the kernels will be obtained.

\bigskip

A very interesting problem in the theory of multivariate orthogonal polynomials is that of
finding asymptotic estimates for the Christoffel functions, because these estimates are related
to the convergence of the Fourier series. Asymptotics for Christoffel functions associated to the
classical orthogonal polynomials on the ball were obtained by Y. Xu in 1996 (see \cite{Xu1996}).
Recently, more general results on the asymptotic behaviour of the Christoffel functions were
established by Kro\'o and Lubinsky \cite{KrooLubinsky2013A,KrooLubinsky2013B}. Those results
include estimates in a  quite general case where the orthogonality measure satisfies some regularity
conditions.

Since our orthogonal polynomials do not fit into the above mentioned case, the asymptotic of the
Christoffel functions deserves special attention. Not surprisingly, our results show that
in any compact subset of the interior of the unit ball Christoffel functions in the Sobolev case
behave exactly as in the classical case, see Theorem~\ref{th2}. On the sphere the situation
is quite different and we can perceive the influence of the outward normal derivatives
in the inner product, see Theorem~\ref{th1}.

\bigskip

The paper is organized as follows. In the next section, we state the background materials on orthogonal
polynomials on the unit ball and spherical harmonics that we will need later.
In Section 3, using spherical polar coordinates we construct explicitly a sequence
of mutually orthogonal polynomials with respect to $ \langle \cdot ,\cdot \rangle_{\mu}^S$.
Those polynomials are given in terms of spherical harmonics and a family of univariate Sobolev
orthogonal polynomials in the radial part, their properties are studied in Section 4.
In Section 5, we deduce explicit connection formulas relating classical multivariate
ball polynomials and our family of Sobolev orthogonal polynomials.
Moreover, an explicit representation for the kernels is obtained. The asymptotic
behaviour of the corresponding Christoffel functions is studied in Section 6.
And finally, in Section 7, we consider the special case $d=2$.

\section{Preliminaries}
\setcounter{equation}{0}

In this section we describe background materials on orthogonal polynomials and spherical harmonics.
The first subsection is devoted to recall some properties on the Jacobi polynomials that
we shall need later. Second subsection recalls the basic results on spherical harmonics
and classical orthogonal polynomials on the unit ball.

\subsection{Classical Jacobi polynomials}\label{Jacobi}

First, we collect some properties of classical Jacobi polynomials $P_n^{(\alpha,\beta)}(t)$.
All of them are well known and can be found in \cite[Chapt. 22]{AbramowitzStegun1964} and \cite{Szego1975}.
For $\alpha, \beta > -1$, these
polynomials are orthogonal with respect to the Jacobi inner product
$$
\left(f,g\right)_{[\alpha,\beta]}=\int_{-1}^1f(t)\, g(t)\, w_{\alpha,\beta}(t)dt,
$$
where the weight function is defined as
$$
   w_{\alpha,\beta}(t) = (1-t)^\alpha(1+t)^{\beta}, \qquad  -1< t < 1.
$$
Jacobi polynomials are normalized by
\begin{equation} \label{jac-norm}
P_n^{(\alpha,\beta)}(1) = \binom{n+\alpha}{n}.
\end{equation}
The squares of the $L^2$ norms are expressed as
\begin{equation}\label{normJ}
h_n^{(\alpha,\beta)} =  \left(P_{n}^{(\alpha, \beta)},P_{n}^{(\alpha, \beta)}\right)_{[\alpha,\beta]}
= \frac{2^{\alpha+\beta+1}\,\Gamma(n+\alpha+1)\,\Gamma(n+\beta+1)}{(2n+\alpha+\beta+1)\,n!\,
\Gamma(n+\alpha+\beta+1)}.
\end{equation}
The polynomial $P_n^{(\alpha,\beta)}(t)$ is of degree $n$ and its leading coefficient $k_n^{(\alpha,\beta)}$
is given by
\begin{equation}\label{leadingcoef}
k_n^{(\alpha,\beta)} = \frac{1}{2^n}\, \binom{2n + \alpha + \beta}{n}.
\end{equation}
The derivative of a Jacobi polynomial is again a Jacobi {polynomial},
\begin{equation}\label{derJ}
\frac{d}{d t} P_n^{(\alpha,\beta)}(t) = \frac{n+\alpha + \beta+1}{2} P_{n-1}^{(\alpha+1,\beta+1)}(t).
\end{equation}

The following relation between different families of the Jacobi polynomials also hold:
\begin{equation}\label{JacAF2}
P_n^{(\alpha,\beta)}(t) = \frac{n+\alpha+\beta+1}{2\,n + \alpha + \beta+1} P_n^{(\alpha+1,\beta)}(t)
- \frac{n+\beta}{2\,n + \alpha + \beta+1}P_{n-1}^{(\alpha+1,\beta)}(t).
\end{equation}
As usual we will denote by $p_n^{(\alpha,\beta)}(t)$ the orthonormal Jacobi polynomial of degree $n$.
Moreover, using (\ref{jac-norm}), (\ref{normJ}) and (\ref{derJ}), we get
\begin{eqnarray}
&& \hspace*{-1.3cm} p_{n}^{\left( \alpha ,\beta \right) }\left( 1\right) =\left( \frac{2n+\alpha
+\beta +1}{2^{\alpha +\beta +1}}\frac{\Gamma \left( n+\alpha +1\right)
\Gamma \left( n+\alpha +\beta +1\right) }
{\Gamma \left( n+1\right) \left( n+\beta +1\right) }\right)^{1/2}\nonumber
\\
&& \times \Gamma
\frac{1}{\Gamma \left( \alpha
+1\right) },\label{JacON1}
\\[10pt]
&& \hspace*{-1.3cm} p_{n}^{\left( \alpha ,\beta \right) \prime }\left( 1\right) = \left( \frac{
2n+\alpha +\beta +1}{2^{\alpha +\beta +3}}\frac{\Gamma \left( n+\alpha
+1\right) \Gamma \left( n+\alpha +\beta +1\right)n}{\Gamma \left( n\right)
\Gamma \left( n+\beta +1\right) }\right)^{1/2}\nonumber
\\
&&\qquad \times \frac{n+\alpha +\beta +1}{
\Gamma \left( \alpha +2\right) }.\label{JacON2}
\end{eqnarray}

In addition to the Jacobi polynomials we will use the corresponding kernel polynomials
defined as
\begin{equation}\label{KernelJac}
K_n(t,u;\alpha, \beta) = \sum_{k=0}^n \, \frac{P_k^{(\alpha,\beta)}(t)\,
P_k^{(\alpha,\beta)}(u)}{h_k^{(\alpha,\beta)}},
\end{equation}
which are symmetric functions. When it is clear from the context, we will omit the parameters
$\alpha$ and $\beta$ in the notation. We also denote the partial derivatives
$$
K_n^{(0,1)}(t,u) = \frac{\partial}{\partial u}K_n(t,u), \qquad
K_n^{(1,1)}(t,u) = \frac{\partial^2}{\partial t\,\partial u}K_n(t,u).
$$
It is well known (see \cite[p. 71]{Szego1975}) that
\begin{equation}\label{K(t,1)}
K_{n}(t,1) = \frac{2^{-\alpha-\beta-1}}{\Gamma(\alpha+1)}\frac{\Gamma(n+\alpha+\beta+2)}{\Gamma(n+\beta+1)}\,
P_{n}^{(\alpha+1,\beta)}(t).
\end{equation}

\bigskip

On the other hand, taking derivatives in the Christoffel--Darboux formula for the kernels in
\cite[(4.5.2) p. 71]{Szego1975},
and expressing the derivative of the kernel in terms of the Jacobi polynomials of parameters $(\alpha+2, \beta)$,
it can be shown that
\begin{eqnarray}
\lefteqn{K_{n}^{(0,1)}(t,1) = 2^{-\alpha-\beta-2}\frac{\Gamma(n+\alpha+\beta+3)}{\Gamma(\alpha+2)\Gamma(n+\beta+1)}
}\nonumber\\ &
\times & \left(\frac{n(n+\alpha+\beta+1)}{2n+\alpha+\beta+2}P_{n}^{(\alpha+2,\beta)}(t) \right.
\left. -\frac{(n+1)(n+\beta)}{2n+\alpha+\beta+2}P_{n-1}^{(\alpha+2,\beta)}(t)\right).\label{K01(t,1)}
\end{eqnarray}

\bigskip

In this way, we can compute the values of the kernels at the point $(1,1)$.

\begin{lemma}\label{Kn}
For $n\geqslant 0$, we get
\begin{eqnarray*}
K_{n}(1,1) &=& \frac{2^{-\alpha-\beta-1}}{\Gamma(\alpha+1)}
\frac{\Gamma(n+\alpha+\beta+2)}{\Gamma(n+\beta+1)}\frac{\Gamma(n+\alpha+2)}{\Gamma(n+1)\,
\Gamma(\alpha+2)},\\
K_{n}^{(0,1)}(1,1)  &=&  \frac{2^{-\alpha-\beta-2}}{\Gamma(\alpha+1)}
\frac{\Gamma(n+\alpha+\beta+3)}{\Gamma(n+\beta+1)} \frac{\Gamma(n+\alpha+2)}{\Gamma(n)\,
\Gamma(\alpha+3)},\\
K_{n}^{(1,1)}(1,1)  &=&  \frac{2^{-\alpha-\beta-3}}{\Gamma(\alpha+2)}\frac{\Gamma(n+\alpha+\beta+3)}
{\Gamma(n+\beta+1)} \frac{\Gamma(n+\alpha+2)}{\Gamma(n)\,\Gamma(\alpha+4)}\\
& & \times \,((\alpha+2)n(n+\alpha+\beta+2) + \beta).
\end{eqnarray*}
\end{lemma}

\bigskip

\subsection{Orthogonal polynomials on the unit ball and spherical harmonics}

For a multi--index $\kappa = (\kappa_1, \ldots, \kappa_d)\in\mathbb{N}_0^d$, and
$x = (x_1, \ldots, x_d)$, a monomial in the variables
$x_1, \ldots, x_d$ is a product
$$
x^{\kappa} = x_1^{\kappa_1} \ldots x_d^{\kappa_d}.
$$
The number $|\kappa| = \kappa_1 + \ldots + \kappa_d$ is called the total degree of $x^{\kappa}$.
A polynomial $P$ in $d$ variables is a finite linear combination of monomials.

Let $\Pi^d$ denote the space of polynomials in $d$ real variables. For a given
non negative integer $n$,
let $\Pi_n^d$ denote the linear space of polynomials in several variables of total degree
at most $n$, and let $\mathcal{P}_n^d$ denote the space of homogeneous polynomials of degree
$n$. It is well known that
$$
\dim \Pi_n^d = \binom{n+d}{n} \quad \hbox{and} \quad \dim \mathcal{P} _n^d =
\binom{n+d-1}{n} = r_n^d.
$$

For $x,y  \in \mathbb{R}^d$, we use the standard notation of $\|x\|$ for the
Euclidean norm of $x$, and $ \langle x, y\rangle$ for the Euclidean product of $x$ and $y$.  The unit ball and the
unit sphere in $\mathbb{R}^d$ are
denoted, respectively, by
$$
\mathbb{B}^d =\{x\in \mathbb{R}^d: \|x\| \leqslant   1\} \qquad \textrm{and} \qquad \mathbb{S}^{d-1}=
\{\xi\in \mathbb{R}^d: \|\xi\| = 1\}.
$$
For $\mu \in \mathbb{R}$, let $W_\mu$ be the weight function defined by
$$
    W_\mu(x) = (1-\|x\|^2)^\mu, \qquad  \|x\| < 1.
$$
The function $W_\mu$ is integrable on the unit ball if $\mu > -1$, for which we denote the normalization
constant of $W_\mu$ by
\begin{equation}\label{omegamu}
\omega_\mu = \int_{{\mathbb{B}^d}}\, W_\mu(x) \, dx = \frac{\pi^{d/2}\Gamma(\mu+1)}{\Gamma(\mu + d/2 + 1)}.
\end{equation}
The weight $W_\mu$ is a radial and centrally symmetric function, that is,
$W_\mu(-x) = W_\mu(x)$, for all $x\in \mathbb{B}^d$.

Let us consider the classical inner product on the unit ball
$$
\langle f,g \rangle_\mu = \frac{1}{\omega_\mu} \int_{{\mathbb{B}^d}}\, f(x)\, g(x) \, W_\mu(x) \, dx,
$$
which is normalized so that $\langle 1,1\rangle_\mu = 1$.

A polynomial $P \in \Pi_n^d$ is called orthogonal with respect to $W_\mu$ on the ball if
$\langle P, Q\rangle_\mu =0$ for all $Q \in \Pi_{n-1}^d$.
Let $\mathcal{V} _n^d(W_\mu)$ denote the linear space of orthogonal polynomials of total
degree $n$ with respect to $W_\mu$. Then $\dim \, \mathcal{V}_n^d(W_\mu) = r_n^d.$

For $n\geqslant 0$, let $\{P^n_{\nu}(x) : |\nu|=n\}$ denote a basis of $\mathcal{V}_n^d(W_\mu)$.
Notice that every element of $\mathcal{V} _n^d(W_\mu)$ is orthogonal to polynomials of lower degree. If the
elements of the basis are also orthogonal to each other, that is,
$\langle P_\nu^n, P_\eta^n \rangle_\mu=0$ whenever $\nu \ne \eta$,
we call the basis mutually orthogonal. If, in addition,
$\langle P_\nu^n, P_\nu^n \rangle_\mu =1$, we call the basis orthonormal.

Since the weight function $W_\mu(x)$ is centrally symmetric, then an orthogonal polynomial on the ball
of degree $n$ is a sum of monomials of even degree if $n$ is even, and sum of monomials of odd degree if
$n$ is odd (\cite[p. 78]{DunklXu2014}).

\bigskip

Harmonic polynomials of degree $n$ in $d$--variables are polynomials in $\mathcal{P} _n^d$ that satisfy
the Laplace equation $\Delta Y = 0$, where
$$\Delta = \frac{\partial^2}{\partial x_1^2} + \ldots + \frac{\partial^2}{\partial x_d^2}$$
is the usual Laplace operator.

If $Y(x)$ is a harmonic polynomial of degree $n$, by Euler's equation for homogeneous polynomials,
we deduce
\begin{equation} \label{Euler}
\langle x, \nabla\rangle Y(x) = \sum_{i=1}^d x_i \frac{\partial}{\partial x_i} Y(x) = n Y(x).
\end{equation}

Let $\mathcal{H}_n^d$ denotes the space of harmonic polynomials
of degree $n$. It is well known that
$$
         a_n^d = \dim \mathcal{H}_n^d = \binom{n+d-1}{n} - \binom{n+d-3}{n-2}.
$$
Spherical harmonics are the restriction of harmonic polynomials to the unit sphere. If $Y \in
\mathcal{H}_n^d$, {then} in spherical--polar {coordinates}
$x = r \xi$, $ r = \|x\| \geqslant 0$, and $\xi \in \mathbb{S}^{d-1}$, we get $Y(x) = r^n Y(\xi)$, so that
$Y$ is uniquely determined by its restriction to the sphere. We shall also use $\mathcal{H}_n^d$ to
denote the space of spherical harmonics of degree $n$.

Let $d\sigma$ denote the surface measure on $\mathbb{S}^{d-1}$ and let $\sigma_d$ denote the surface area,
\begin{equation}\label{sigmad}
  \sigma_d = \int_{\mathbb{S}^{d-1}} d\sigma = \frac{2\, \pi^{d/2}}{\Gamma(d/2)}.
\end{equation}
Using Green's formula on the sphere it is easy to see that spherical harmonics of different degrees are
orthogonal with respect to the inner product
$$
\langle f, g \rangle_{\mathbb{S}^{d-1} } = \frac{1}{\sigma_d} \int_{\mathbb{S}^{d-1} } f(\xi) g(\xi)
d\sigma(\xi).
$$

\bigskip

In spherical--polar coordinates a mutually orthogonal
basis of $\mathcal{V}_n^d(W_\mu)$ can be given in terms of the Jacobi polynomials and spherical harmonics
(see for example, \cite{DunklXu2014}).

\begin{lemma}\label{opb}
For $n \in \mathbb{N}_0$ and $0 \leqslant   j \leqslant   n/2$, let $\{Y_\nu^{n-2j}(x): 1\leqslant   \nu
\leqslant   a_{n-2j}^d\}$ denote
an orthonormal basis for $\mathcal{H}_{n-2j}^d$. Define
\begin{equation}\label{baseP}
P_{j,\nu}^{n}(x;\mu) = P_{j}^{(\mu, \beta^n_j)}(2\,\|x\|^2 -1)\, Y_\nu^{n-2j}(x),
\end{equation}
where $\beta^n_j = n-2j + \frac{d-2}{2}$.

Then the set $\{P_{j,\nu}^{n}(x;\mu): 0 \leqslant   j \leqslant   n/2, \,1 \leqslant   \nu \leqslant   a_{n-2j}^d \}$
is a mutually
orthogonal basis of $\mathcal{V} _n^d(W_\mu)$.

More precisely,
$$
\langle P_{j,\nu}^{n}, P_{k,\eta}^{m}\rangle_\mu =  H_{j,n}^{\mu}  \delta_{n,m}\,\delta_{j,k}\,
\delta_{\nu,\eta},
$$
where $ H_{j,n}^{\mu}=\langle P_{j,\nu}^{n}, P_{j,\nu}^{n}\rangle_\mu$ is given by
$$
 H_{j,n}^{\mu} = \frac{1}{2^{\mu + \beta^n_j + 2}}\,\frac{\sigma_d}{\omega_\mu}\, h^{(\mu, \beta^n_j)}_j.
$$
\end{lemma}

\bigskip

\section{A Sobolev inner product on the ball}
\setcounter{equation}{0}

Let us define the Sobolev inner product
\begin{equation}\label{SBIP}
\langle f, g \rangle_\mu^S = \frac{1}{\omega_\mu} \int_{\mathbb{B}^d} f(x) g(x) W_\mu(x)\, dx +
\frac{\lambda}{\sigma_{d}} \int_{\mathbb{S}^{d-1}} \frac{\partial f}{\partial \mathbf{n}}(\xi)
\frac{\partial g}{\partial \mathbf{n}}(\xi) d\sigma(\xi),
\end{equation}
where $W_\mu(x)=(1-\|x\|^2)^\mu$, $\mu > -1$, is the classical weight function on the ball, $\omega_\mu$
and  $\sigma_{d}$
are given by (\ref{omegamu}) and (\ref{sigmad}), respectively, and $\frac{\partial}{\partial \mathbf{n}}$
stands for the outward normal derivative operator, which on the sphere $\mathbb{S}^{d-1}$ is given by
$$
\frac{\partial f}{\partial \mathbf{n}} = \sum_{i=1}^d x_i \frac{\partial f}{\partial x_i}.
$$

We observe that the above inner product is centrally symmetric, in the sense that
$\langle x^\kappa, x^\tau \rangle_\mu^S = 0$ whenever $|\kappa|+|\tau|$ odd. This implies that
an orthogonal polynomial of degree $n$ is a sum of monomials of even degree
if $n$ is even, and a sum of monomials of odd degree if $n$ is odd.

In next theorem we will construct a mutually orthogonal basis relative to the previous
Sobolev inner product, which will be given explicitly in terms of spherical harmonics and
a family of Sobolev orthogonal polynomials in one variable.

\begin{theorem}
Let $\{q_j^{(\alpha,\beta;M)}(t)\}_{j\geqslant 0}$ denote
the univariate Sobolev orthogonal po\-ly\-no\-mials orthogonal with respect to the Sobolev inner product
\begin{equation}\label{eq:op1d}
(f,g)_{[\alpha,\beta;M]}^S = \int_{-1}^{1} f(t) g(t) (1-t)^\alpha (1+t)^\beta\, dt + {\bf f}(1)  M  {\bf g}(1)^{t},
\end{equation}
where ${\bf f}(1) = (f(1), f'(1))$ and $M$ is a $2\times 2$ symmetric positive semidefinite matrix.
Let $\{Y_{\nu}^{n-2j}(x) \, : \, 1\leqslant  \nu\leqslant   a_{n-2j}^d\}$ be an orthonormal basis of the  spherical
harmonics $\mathcal{H}_{n-2j}^d$.
Let us define the polynomials in $d$ variables
\begin{equation}\label{eq:sovpol}
Q_{j,\nu}^n(x) = q_j^{(\mu,\beta_j^n;M_{n-2j})}(2\|x\|^2-1) Y_{\nu}^{n-2j}(x),
\end{equation}
with $\beta_j^n=n-2j+\delta$, $\ \delta =  \frac{d-2}{2}$,
\begin{equation}\label{a0}
A_{0}=\lambda \,2^{\delta +\mu +2}\frac{\omega _{\mu }}{\sigma _{d}},
\end{equation}
and
\begin{equation}\label{defM}
M_{n-2j} = 2^{n-2j}\,A_0\,\left[\begin{array}{cc}(n-2j)^2 & 4 (n-2j) \\
4 (n-2j) &16\end{array}\right].
\end{equation}
Then, for $n\geqslant0$, the set $\{Q_{j,\nu}^n(x)\,:\,0\leqslant   j\leqslant   n/2,\,\,1\leqslant  \nu\leqslant   a_{n-2j}^d\}$ is
a mutually orthogonal basis of $\mathcal{V}_n^d(W_\mu,S)$, the linear space of polynomials of degree $n$
which are orthogonal with respect to the Sobolev inner product (\ref{SBIP}).

Moreover,
\begin{equation}\label{tilde--norm}
\tilde{H}^\mu_{j,\nu} = \langle Q_{j,\nu}^n, Q_{j,\nu}^n \rangle_\mu^S = \frac{\lambda}{2^{n-2j}\,A_0}\,
\tilde{h}^{(\mu,\beta_j^n;M_{n-2j})}_j,
\end{equation}
where\quad $\tilde{h}_j^{(\mu,\beta;M)} = (q_j^{(\mu,\beta;M)},\,q_j^{(\mu,\beta;
M)})_{[\mu,\beta;M]}^S.$
\end{theorem}

\bigskip

\begin{proof}
In order to check the orthogonality, we need to compute the product
\begin{eqnarray}
\langle Q_{j,\nu}^n, Q_{k,\eta}^m \rangle_\mu^S &=& \frac{1}{\omega_\mu} \int_{\mathbb{B}^d} Q_{j,\nu}^n(x)
Q_{k,\eta}^m(x) W_\mu(x)\, dx \label{pr1}\\
&~& + \frac{\lambda}{\sigma_{d}} \int_{\mathbb{S}^{d-1}} \frac{\partial Q_{j,\nu}^n}{\partial \mathbf{n}}(\xi)
 \frac{\partial Q_{k,\eta}^m}{\partial \mathbf{n}}(\xi) \, d\sigma(\xi). \nonumber
\end{eqnarray}

Let us start with the computation of the first integral.
$$
I_1=\frac{1}{\omega_\mu} \int_{\mathbb{B}^d} Q_{j,\nu}^n(x) Q_{k,\eta}^m(x) W_\mu(x)\, dx.
$$
Using spherical--polar coordinates and the orthogonality of the
spherical harmonics we obtain
\begin{eqnarray*}
I_1 &=& \frac{\sigma_d}{\omega_\mu} \int_0^1 q_j^{(\mu,\beta_j^n)}(2r^2-1)
q_k^{(\mu,\beta_k^m)}(2r^2-1)(1-r^2)^\mu\,  r^{n-2j+m-2k+d-1}  \, dr
\\
& ~&\times \, \delta_{n-2j,m-2k} \delta_{\nu\eta}
\\
&=& \frac{\sigma_d}{\omega_\mu} \int_0^1 q_j^{(\mu,\beta_j^n)}(2r^2-1)
q_k^{(\mu,\beta_j^n)}(2r^2-1) (1-r^2)^\mu\,  r^{2(n-2j)+d-1} \, dr
\\
&~& \times \, \delta_{n-2j,m-2k} \delta_{\nu\eta},
\end{eqnarray*}
where $q_j^{(\mu,\beta_j^n)}\equiv q_j^{(\mu,\beta_j^n;M_{n-2j})}$ and
$q_k^{(\mu,\beta_k^m)} \equiv q_k^{(\mu,\beta_k^m;M_{m-2k})}$.

Finally, the change of variables $t=2r^2-1$ moves the integral to the interval $[-1,1]$,
\begin{eqnarray}
I_1 &=& \frac{1}{2^{\beta_j^n+\mu+2}} \frac{\sigma_d}{\omega_\mu} \int_{-1}^1 q_j^{(\mu,\beta_j^n)}(t)
 q_k^{(\mu,\beta_j^n)}(t) (1-t)^\mu (1+t)^{\beta_j^n}  \, dt  \label{i1}\\
&~& \times \delta_{n-2j,m-2k} \delta_{\nu\eta}. \nonumber
\end{eqnarray}

Let us now compute the second integral in (\ref{pr1}),
$$
I_2 = \frac{\lambda}{\sigma_{d}} \int_{\mathbb{S}^{d-1}}
\frac{\partial Q_{j,\nu}^n}{\partial \mathbf{n}}(\xi) \frac{\partial Q_{k,\eta}^m}{\partial \mathbf{n}}(\xi)
\, d\sigma(\xi).
$$
In order to easily get that integral, we need some previous results.

Computing the normal derivatives
$$
\frac{\partial }{\partial \mathbf{n}} \Big( q_j^{(\mu,\beta;M)}(2\|\xi\|^2-1) \Big)
=
4 \|\xi\|^2 (q_j^{(\mu,\beta;M)})' (2\|\xi\|^2-1),
$$
and using Euler's formula (\ref{Euler}), we deduce
\begin{eqnarray*}
\lefteqn{\frac{\partial }{\partial \mathbf{n}} \Big( q_j^{(\mu,\beta;M)}(2\|\xi\|^2-1) \, Y_{\nu}^{n-2j}(\xi) \Big)
}\\
&=& \Big(4 \|\xi\|^2 (q_j^{(\mu,\beta;M)})' (2\|\xi\|^2-1) + (n-2j) q_j^{(\mu,\beta;M)}(2\|\xi\|^2-1)\Big)\,
 Y_{\nu}^{n-2j}(\xi).
\end{eqnarray*}

Thus, the second integral splits into four terms,
\begin{eqnarray*}
I_2
&=& \frac{\lambda}{\sigma_{d}}
\Big(16 \,q_j'(1) \,q_k'(1) + 4\, (n-2j) \,
q_j'(1) \,q_k(1)
 + 4 \,(n-2j) \,q_j(1) \,q_k'(1) \\
&~& + (n-2j)^2\,
q_j(1) \,q_k(1)  \Big) \,\int_{\mathbb{S}^{d-1}} Y^{n-2j}_{\nu}(\xi)\, Y^{n-2k}_{\eta}(\xi)\, d\sigma(\xi) \\
&=& \lambda \,\Big( 16\, (q_j'(1)\, q_k'(1) +
4 \,(n-2j) \,q_j'(1) \,q_k(1)
+ 4 \,(n-2j) \,q_j(1) \,q_k'(1) \\
& ~&+
(n-2j)^2 \,q_j(1) \,q_k(1)  \Big)
\, \delta_{n-2j,m-2k} \, \delta_{\nu,\eta},
\end{eqnarray*}
where we have omitted the superscript in $q_j^{(\mu,\beta_j^n;M_{n-2j})}$ for brevity.

Finally, this can be written in matrix form as follows
\begin{equation} \label{i2}
I_2 = \lambda \mathbf{q}_j^{(\mu,\beta_j^n;M_{n-2j})}(1 )\, \tilde{M}_{n-2j} \,
\mathbf{q}_k^{(\mu,\beta_j^n;M_{n-2j})}(1)^{t}
\end{equation}
where $\mathbf{q}_j^{(\mu,\beta_j^n;M_{n-2j})}(1 ) = \Big(q_j^{(\mu,\beta_j^n;M_{n-2j})}(1),
(q_j^{(\mu,\beta_j^n;M_{n-2j})})'(1)\Big)$
and
$$ \tilde{M}_{n-2j} = \left[
\begin{array}{cc} (n-2j)^2 & 4(n-2j) \\ 4(n-2j) & 16\end{array}\right].
$$
Observe that $M_{n-2j}=2^{n-2j}\,A_0\,\tilde{M}_{n-2j}$.

To end the proof, we just have to take together (\ref{i1}) and (\ref{i2}) to get the value of (\ref{pr1})
in terms of the Sobolev inner product (\ref{eq:op1d}) as
\begin{eqnarray*}
\langle Q_{j,\nu}^n, Q_{k,\eta}^m \rangle_\mu^S &=& \frac{\lambda}{2^{n-2j}\,A_0}\,
\Big(q_j^{(\mu,\beta_j^n;M_{n-2j})},\,q_k^{(\mu,\beta_j^n;M_{n-2j})}\Big)_{[\mu,\beta_j^n;M_{n-2j}]}^S \\
&~& \times \delta_{n-2j,m-2k}
\, \delta_{\nu,\eta}.
\end{eqnarray*}
Then the result follows from the orthogonality of the univariate Sobolev orthogonal polynomials.

\end{proof}

\bigskip

\section{The univariate non--diagonal Sobolev inner product}
\setcounter{equation}{0}

In this section we will explore some properties of the univariate Sobolev orthogonal polynomials involved
in (\ref{eq:sovpol}).

Let $(\cdot,\cdot)_{[\alpha,\beta;M]}^S$ be the non--diagonal Sobolev inner product defined in (\ref{eq:op1d})
by
$$
(f,g)_{[\alpha,\beta;M]}^S = \int_{-1}^1 f(t)\, g(t) (1-t)^\alpha (1+t)^\beta dt +
\mathbf{f}(1)  M  \mathbf{g}(1)^{t},
$$
where $M$ is a positive semidefinite matrix and $\mathbf{f}(1)=(f(1), f'(1))$.

Let $\{q_n^{(\alpha, \beta;M)}(t)\}_{n\geqslant 0}$ be the orthogonal polynomials with respect to this inner
product, normalized with leading coefficient
$k_n^{(\alpha, \beta)}$ given in (\ref{leadingcoef}). Some properties for the monic orthogonal polynomials
with respect to this inner product can be found in \cite{AlfaroMarcellanRezolaRonveaux1995,DuenasGarza2013}.

In what follows, when not confusing, we will simplify the notations $q_n^{(\alpha, \beta;M)}\equiv
q_n^{(\alpha, \beta)}\equiv q_n$, and $(f,g)_{[\alpha,\beta;M]}^S\equiv(f,g)^S$.

These univariate Sobolev orthogonal polynomials can be expressed in terms of the classical Jacobi polynomials as follows.

\begin{lemma}\label{S1v}
For $\alpha,\beta>-1$, it holds
\begin{equation}\label{relq13}
q_j^{(\alpha, \beta)}(t)= b_{j,j}^{(\alpha, \beta)}\, P_j^{(\alpha+2, \beta)}(t)+b_{j,j-1}^{(\alpha, \beta)}
\, P_{j-1}^{(\alpha+2, \beta)}(t)+b_{j,j-2}^{(\alpha, \beta)}\, P_{j-2}^{(\alpha+2, \beta)}(t),
\end{equation}
where
\begin{eqnarray*}
b_{j,j}^{(\alpha, \beta)}&=&\frac{(j+\alpha+\beta+2)(j+\alpha+\beta+1)}{(2j+\alpha+\beta+2)(2j+\alpha+\beta+1)}, \\
b_{j,j-1}^{(\alpha, \beta)}&=& \frac{(j+\alpha+\beta+1)}{2j+\alpha+\beta}
\left( -\frac{2(j+\beta)}{2j+\alpha+\beta+2}-c^j_{j-1}\right),\\
b_{j,j-2}^{(\alpha, \beta)} &=&\frac{j+\beta-1}{2j+\alpha+\beta}\left(\frac{j+\beta}
{2j+\alpha+\beta+1}+c^j_{j-2}\right),
\end{eqnarray*}
with
\begin{eqnarray*}
c^j_{j-1}&=&2^{-\alpha-\beta-2}\frac{\Gamma(j+\alpha+\beta+1)}
{\Gamma(\alpha+1)\Gamma(j+\beta)}\mathbf{P}_j(1)\, \Lambda_{j-1} \left[\begin{array}{cc} 2 \\
(j-1)(j+\alpha+\beta)
\end{array} \right],\\
c^j_{j-2}&=&2^{-\alpha-\beta-2}\frac{\Gamma(j+\alpha+\beta+1)}
{\Gamma(\alpha+1)\Gamma(j+\beta)}\mathbf{P}_j(1) \, \Lambda_{j-1}  \left[\begin{array}{cc} 2 \\ j
(j+\alpha+\beta+1)\end{array}\right],
\end{eqnarray*}
and
\begin{eqnarray}
&&\mathbf{P}_j(1) \equiv\mathbf{P}_j^{(\alpha,\beta)}(1) = \left(P_j^{(\alpha,\beta)}(1),
(P_j^{(\alpha,\beta)})'(1)\right),
 \nonumber\\
&&\Lambda_{j-1} \equiv\Lambda^{(\alpha,\beta;M)}_{j-1} = \left(I+M \mathcal{K}_{j-1}\right)^{-1} M,
\label{defLambda}\\
&&\mathcal{K}_{j-1} \equiv \mathcal{K}^{(\alpha,\beta)}_{j-1}= \left[
\begin{array}{cc}
K_{j-1}(1,1) & K_{j-1}^{(1,0)}(1,1)\\
K_{j-1}^{(0,1)}(1,1) & K_{j-1}^{(1,1)}(1,1)
\end{array}\right]. \label{kappa}
\end{eqnarray}
\end{lemma}

\bigskip

\begin{proof} If we expand $q_j^{(\alpha,\beta)}$ in terms of Jacobi polynomials,
$$
q_j^{(\alpha,\beta)}(t)=\sum_{i=0}^j b_{j,i}^{(\alpha,\beta)} P_{i}^{(\alpha+2, \beta)}(t),
$$
and using standard techniques
$$b_{j,i}^{(\alpha,\beta)}  = \frac{(q_j^{(\alpha,\beta)}(t),
P_{i}^{(\alpha+2, \beta)}(t))_{[\alpha+2,\beta]}}{h_{i}^{(\alpha+2, \beta)}}=
\frac{(q_j^{(\alpha,\beta)}(t), P_{i}^{(\alpha, \beta)}(t)\,
(1-t)^2)^S}{h_{i}^{(\alpha+2, \beta)}}.$$
Then, for $i<n-2$, $b_{j,i}^{(\alpha,\beta)}=0$, and so relation (\ref{relq13}) holds.
Moreover, the coefficient
$b_{j,j}^{(\alpha,\beta)}$ can be determined using the leading coefficients of the polynomials $q_j^{(\alpha,
\beta)}(t)$ and $P_j^{(\alpha+2, \beta)}(t)$ both given by (\ref{leadingcoef})
$$
b_{j,j}^{(\alpha,\beta)}=\frac{k_j^{(\alpha,\beta)}}{k_j^{(\alpha+2,\beta)}}=\frac{(j+\alpha+\beta+2)
(j+\alpha+\beta+1)}{(2j+\alpha+\beta+2)(2j+\alpha+\beta+1)}.
$$
We determine the other two coefficients using Proposition 2 in \cite{AlfaroMarcellanRezolaRonveaux1995},
\begin{equation}\label{MM}
q_j^{(\alpha,\beta)}(t)=P_j^{(\alpha,\beta)}(t)- \mathbf{P}_j^{(\alpha,\beta)}(1) \,\Lambda_{j-1}
\mathbf{K}^{\alpha,\beta}_{j-1}(t,1),
\end{equation}
where
\begin{equation}\label{Kalpha}
\mathbf{K}^{\alpha,\beta}_{j-1}(t,1)=\left[
\begin{array}{c}
K_{j-1}(t,1;\alpha,\beta) \\
K_{j-1}^{(0,1)}(t,1;\alpha,\beta)
\end{array}\right].
\end{equation}
If we apply equation (\ref{JacAF2}) twice we obtain
\begin{eqnarray*}
P_j^{(\alpha,\beta)}(t)=\frac{(j+\alpha+\beta+1)(j+\alpha+\beta+2)}{(2j+\alpha+\beta+1)(2j+\alpha+\beta+2)}
\, P_j^{(\alpha+2,\beta)}(t) \\
-\frac{2(j+\alpha+\beta+1)(j+\beta)}{(2j+\alpha+\beta+2)(2j+\alpha+\beta)}\, P_{j-1}^{(\alpha+2,\beta)}(t) \\
+ \frac{(j+\beta)(j+\beta-1)}{(2j+\alpha+\beta+1)(2j+\alpha+\beta)}\, P_{j-2}^{(\alpha+2,\beta)}(t).
\end{eqnarray*}
Substituting in (\ref{MM}), and using (\ref{K(t,1)}) and (\ref{K01(t,1)}), the result follows.

\end{proof}

\bigskip

For $n\geqslant 0$, we denote by
$$
\tilde{K}_n(t,u)=\sum_{j=0}^n \frac{q_j^{(\alpha,\beta)}(t)\,q_j^{(\alpha,\beta)}(u)}{
\tilde{h}_j^{(\alpha,\beta)}}
$$
the reproducing kernels associated with the polynomials $q_j^{(\alpha,\beta)}(t)$.

We need to establish a relationship between these kernels and
the kernels of Jacobi polynomials defined in Section \ref{Jacobi}. To this end we need the following lemmas.

\begin{lemma} \label{lambda}
The matrix $\Lambda_j=\left(I+M \mathcal{K}_{j}\right)^{-1}\,M$ is symmetric.
Moreover,
$$
\Lambda_{j-1} \mathbf{P}_j(1)^t h_j^{-1} \mathbf{P}_j(1) \Lambda_j=\Lambda_{j-1} - \Lambda_j,
$$
where $h_j = h_j^{(\alpha, \beta)}$ is given in (\ref{normJ}).
\end{lemma}

\bigskip

\begin{proof}
Using Sherman--Morrison--Woodbury identity (see \cite{GolubVanLoan2013}), we get
$$
(I+M \mathcal{K}_{j})^{-1}=I-M(\mathcal{K}_{j}^{-1}+M)^{-1},
$$
thus
$$
(I+M \mathcal{K}_{j})^{-1}M = M-M(\mathcal{K}_{j}^{-1}+M)^{-1}M,
$$
and the symmetry of $\Lambda_j$ follows from the symmetry of $M$ and $\mathcal{K}_{j}^{-1}$.
On the other hand,
\begin{eqnarray*}
\lefteqn{\Lambda_{j-1} \mathbf{P}_j(1)^t h_j^{-1} \mathbf{P}_j(1) \Lambda_j =} \\
&=& \left(I+M \mathcal{K}_{j-1}\right)^{-1} (M\mathcal{K}_{j}-M\mathcal{K}_{j-1})  \left(I+M
\mathcal{K}_{j}\right)^{-1}\,M \\
&=& \left(I+M \mathcal{K}_{j-1}\right)^{-1} ((I+M\,\mathcal{K}_{j})-(I+M\mathcal{K}_{j-1}))
\left(I+M \mathcal{K}_{j}\right)^{-1}M \\
&=& \left(I+M \mathcal{K}_{j-1}\right)^{-1}\,M-\left(I+M \mathcal{K}_{j}\right)^{-1}\,M.
\end{eqnarray*}

\end{proof}

\bigskip

\begin{lemma}\label{hinvers} For $j\geqslant 1$,
$$
(\tilde{h}_j^{(\alpha,\beta)})^{-1}=(h_j^{(\alpha,\beta)})^{-1}- (h_j^{(\alpha,\beta)})^{-2}
\, \mathbf{P}^{(\alpha,\beta)}_j(1) \Lambda_{j}\mathbf{P}^{(\alpha,\beta)}_j(1)^t.
$$
\end{lemma}

\bigskip

\begin{proof}

First, we get the relation between the norms, using (\ref{MM})
$$
\tilde{h}_j^{(\alpha,\beta)}=(q_j,q_j)^S=(q_j,p_j)^S=h_j^{(\alpha,\beta)}+\mathbf{q}^{(\alpha,\beta)}_j(1)
M \mathbf{P}^{(\alpha,\beta)}_j(1)^t.
$$
Taking into account that $\mathbf{q}^{(\alpha,\beta)}_j(1)=\mathbf{P}^{(\alpha,\beta)}_j(1) \left(I+M
\mathcal{K}_{j-1}\right)^{-1}$, we get
$$
\tilde{h}_j^{(\alpha,\beta)}=h_j^{(\alpha,\beta)}+\mathbf{P}^{(\alpha,\beta)}_j(1) \,
\Lambda_{j-1} \,\mathbf{P}^{(\alpha,\beta)}_j(1)^t.
$$
On the other hand, from
$$
M\mathbf{P}^{(\alpha,\beta)}_j(1)^t h_j^{-2} \mathbf{P}^{(\alpha,\beta)}_j(1)=
 h_j^{-1}M(\mathcal{K}_{j}-\mathcal{K}_{j-1})=h_j^{-1} [(I+M\mathcal{K}_{j})-
(I+M \mathcal{K}_{j-1})],
$$
we get
\begin{eqnarray*}
\lefteqn{\hspace{-1.0cm}\mathbf{P}^{(\alpha,\beta)}_j(1) \left(I+M\, \mathcal{K}_{j-1}\right)^{-1}\,M\, \mathbf{P}^{(\alpha,
\beta)}_j(1)^t\,
h_j^{-2} \,\mathbf{P}^{(\alpha,\beta)}_j(1) \left(I+M \mathcal{K}_{j}\right)^{-1}\,M
\mathbf{P}^{(\alpha,\beta)}_j(1)^t}\\
&=& h_j^{-1} \,\mathbf{P}^{(\alpha,\beta)}_j(1) \left(I+M \mathcal{K}_{j-1}\right)^{-1}\,
((I+M\,\mathcal{K}_{j})-(I+M\, \mathcal{K}_{j-1})) \\
&~& \times \left(I+M \,\mathcal{K}_{j}\right)^{-1}\,M\,
\mathbf{P}^{(\alpha,\beta)}_j(1)^t \\
&=& h_j^{-1}\,\mathbf{P}^{(\alpha,\beta)}_j(1) \left(I+M \,\mathcal{K}_{j-1}\right)^{-1}\,M
\,\mathbf{P}^{(\alpha,\beta)}_j(1)^t \\
&~& - h_j^{-1}\,\mathbf{P}^{(\alpha,\beta)}_j(1)
\left(I+M\, \mathcal{K}_{j}\right)^{-1}\,M\, \mathbf{P}^{(\alpha,\beta)}_j(1)^t.
\end{eqnarray*}
Then, it is easy to show that
$$
\tilde{h}_j^{(\alpha,\beta)} \left(h_j^{-1}-(h_j^{-1})^2\, \mathbf{P}^{(\alpha,\beta)}_j(1)
\left(I+M \,\mathcal{K}_{j}\right)^{-1}
\,M \,\mathbf{P}^{(\alpha,\beta)}_j(1)^t\right) = 1,
$$
and the result follows.

\end{proof}

\bigskip

Now we are ready to derive an explicit formula for the univariate kernels.

\begin{prop}
For $j\geqslant 0$, we get
\begin{eqnarray}
\lefteqn{\hspace{-1cm}q_j^{(\alpha,\beta)}(t)\,q_j^{(\alpha,\beta)}(u)\,(\tilde{h}_j^{(\alpha,\beta)})^{-1} =
P_j^{(\alpha,\beta)}(t)\,P_j^{(\alpha,\beta)}(u)\,(h_j^{(\alpha,\beta)})^{-1}} \nonumber\\
\hspace{1cm}&-& \mathbf{K}_j(t,1)^t\, \Lambda_j \, \mathbf{K}_j(u,1)+\mathbf{K}_{j-1}(t,1)^t\, \Lambda_{j-1} \,
\mathbf{K}_{j-1}(u,1).\label{KC1v}
\end{eqnarray}
As a consequence, for $n\geqslant 0$,
\begin{equation}\label{K1v}
\tilde{K}_n(t,u) = K_n(t,u) - \mathbf{K}_n(t,1)^t\,\Lambda_n \,\mathbf{K}_n(u,1).
\end{equation}

\end{prop}

\bigskip

\begin{proof}
Using (\ref{MM}) and Lemma \ref{hinvers}, we get
\begin{eqnarray*}
q_j(t)\,q_j(u)\,\tilde{h}_j^{-1}
&=& (P_j(t)- \mathbf{P}_j(1) \,\Lambda_{j-1} \,\mathbf{K}_{j-1}(t,1)) \\
& ~ & \times (P_j(u)- \mathbf{P}_j(1)\, \Lambda_{j-1} \,\mathbf{K}_{j-1}(u,1))
\\
& ~ & \times (h_j^{-1} - (h_j^{-1})^2 \, \mathbf{P}_j(1) \,\Lambda_j \, \mathbf{P}_j(1)^t).
\end{eqnarray*}
Taking into account Lemma~\ref{lambda} and $\mathbf{P}_j(1)h_j^{-1}P_j(u)=\mathbf{K}_j(u,1)^t-
\mathbf{K}_{j-1}(u,1)^t$, we obtain
\begin{eqnarray*}
q_j(t)\,q_j(u)\,\tilde{h}_j^{-1}
&=& P_j(t)\,P_j(u)\, h_j^{-1} -(\mathbf{K}_j(t,1)^t-\mathbf{K}_{j-1}(t,1)^t)\Lambda_{j-1} \mathbf{K}_{j-1}(u,1) \\
&- & (\mathbf{K}_j(t,1)^t-\mathbf{K}_{j-1}(t,1)^t)\Lambda_j (\mathbf{K}_j(u,1)-\mathbf{K}_{j-1}(u,1)) \\
&+&  (\mathbf{K}_j(t,1)^t-\mathbf{K}_{j-1}(t,1)^t)(\Lambda_{j-1}-\Lambda_{j})\mathbf{K}_{j-1}(u,1) \\
&-& (\mathbf{K}_j(u,1)^t-\mathbf{K}_{j-1}(u,1)^t)\Lambda_{j-1}\mathbf{K}_{j-1}(t,1) \\
&+& \mathbf{K}_{j-1}(t,1)^t \Lambda_{j-1} \mathbf{P}_j(1)^t \mathbf{P}_j(1)\Lambda_{j}\mathbf{K}_{j}(u,1).
\end{eqnarray*}
Therefore, we get (\ref{KC1v}) and a telescopic sum gives (\ref{K1v}).

\end{proof}

\bigskip

\section{Multivariate Sobolev orthogonal polynomials}
\setcounter{equation}{0}

In this section, we will express multivariate Sobolev orthogonal polynomials in terms of
classical ball polynomials. To this end, using the following lemmas we will simplify the matrix
 $\Lambda _{m}^{\left( \mu ,k+\delta ;M_{k}\right)}=
( I+M_{k}\mathcal{K}_{m})^{-1}M_{k}$ defined in (\ref{defLambda}), where
$M_{k}$ was introduced in (\ref{defM}),
and $\mathcal{K}_{m}=\mathcal{K}^{(\mu,k+\delta)}_{m}$ was given in (\ref{kappa}).

\bigskip

\begin{lemma}\label{inv}
Let $M$ and $\mathcal{K}$ be $2\times2$ matrices with $\det \left( M\right) =0$.
Then
\[
\left( I+M\mathcal{K}\right) ^{-1}M=\frac{1}{\Delta }M,
\]%
where
$$
\Delta =1 + \mathrm{trace}\left( M\mathcal{K}\right)
$$
is assumed non--zero.
\end{lemma}

\bigskip

\begin{proof}
This is a straightforward calculation.

\end{proof}

\bigskip

\begin{lemma}\label{le2}
Let $k\geqslant 0$, then
$$
\Lambda _{m}^{\left( \mu ,k+\delta;M_{k}\right) } = \frac{1}{\Delta _{k,m}}M_{k},
$$
where
\begin{equation}\label{delta}
\Delta _{k,m}=1+2^{k}A_{0}\left\{ k^{2}K_{m}\left( 1,1\right)
+8kK_{m}^{\left( 1,0\right) }\left( 1,1\right) +16K_{m}^{\left( 1,1\right)
}\left( 1,1\right) \right\} .
\end{equation}
\end{lemma}

\bigskip

\begin{proof}
We see that $\det \left( M_{k}\right) =0$, and
\begin{eqnarray*}
\lefteqn{1+\mathrm{trace}\left( M_{k}\mathcal{K}_{m}\right)} \\
&&\hspace{-0.8cm}=1+2^{k}A_{0}\left\{ k^{2}K_{m}\left( 1,1\right) +4kK_{m}^{\left(
1,0\right) }\left( 1,1\right) +4kK_{m}^{\left( 0,1\right) }\left( 1,1\right)
+16K_{m}^{\left( 1,1\right) }\left( 1,1\right) \right\}.
\end{eqnarray*}
Then Lemma \ref{inv} gives the result.

\end{proof}

\bigskip

If we replace $t=2\,\|x\|^2 -1$ in equation (\ref{relq13}), multiply the result
times $Y_\nu^{n-2j}(x)$, and use (\ref{baseP}) and (\ref{eq:sovpol}),
we can express Sobolev orthogonal polynomials in terms of ball polynomials. This
representation is given in next theorem.

\bigskip

\begin{theorem}
Let $n\in \mathbb{N}_0$, $0\leqslant   j\leqslant   n/2$, and $1\leqslant   \nu\leqslant   a^d_{n-2j}$. Then,
$$
Q^n_{j,\nu}(x)= b_{j,j}^{(\mu, \beta^n_j)}\, P^n_{j,\nu}(x;\mu+2)+b_{j,j-1}^{(\mu, \beta^n_j)}\,
P^{n-2}_{j-1,\nu}(x;\mu+2)\\
+b_{j,j-2}^{(\mu,\beta^n_j)}\, P^{n-4}_{j-2,\nu}(x;\mu+2),
$$
where $P^n_{j,\nu}(x;\mu+2)$ are the polynomials in the ball orthogonal with respect to $W_{\mu+2}(x)$, and
\begin{eqnarray*}
b_{j,j}^{(\mu, \beta^n_j)}&=&\frac{(n-j+\mu+\delta+2)(n-j+\mu+\delta+1)}{(n+\mu+\delta+2)
(n+\mu+\delta+1)}, \\
b_{j,j-1}^{(\mu, \beta^n_j)}&=& \frac{n-j+\mu+\delta+1}{n+\mu+\delta} \Big( -\frac{2(n-j+\delta)
}{n+\mu+\delta+2}-d^n_{j} \, a_{j-1}^{n-2}\Big),\\
b_{j,j-2}^{(\mu, \beta^n_j)} &=&\frac{n-j+\delta-1}{n+\mu+\delta}\Big(\frac{n-j+\delta}{n+\mu+\delta+1}+
d^n_{j} \, a_{j}^n\Big),
\end{eqnarray*}
with
\begin{eqnarray*}
d^n_{j} &=& \frac{A_0\Gamma(n-j+\mu+\delta+1)\,\Gamma(j+\mu+1)}{2^{\mu+\delta+2}\Gamma(\mu+1)\Gamma(\mu+2)
\Gamma(n-j+\delta)\, j!}\\
&~& \times\,((\mu+1)(n-2j)+2j(n-j+\mu+\delta+1)), \\
a_j^n & = & 2\,((n-2j)+2j (n-j+\mu+\delta+1)).
\end{eqnarray*}
\end{theorem}

\bigskip

Let us define the kernels of the ball orthogonal polynomials and the kernels of the Sobolev
orthogonal polynomials
in the usual way,
\begin{eqnarray}
\mathbb{L}_n(x,y) & = & \sum_{m=0}^n \sum_{j=0}^{[m/2]} \sum_{\nu =1}^{a_{m-2j}^d}
P^m_{j,\nu}(x)\, P^m_{j,\nu}(y)
\,({H}^m_{j,\nu})^{-1},\label{KernelBall}\\
\tilde{\mathbb{L}}_n(x,y) & = &\sum_{m=0}^n \sum_{j=0}^{[m/2]} \sum_{\nu =1}^{a_{m-2j}^d}  Q^m_{j,\nu}(x)
\, Q^m_{j,\nu}(y)\,(\tilde{H}^m_{j,\nu})^{-1}.\label{KernelSob}
\end{eqnarray}

Then, we can establish a relation between these kernels by means of the kernels of univariate Jacobi
polynomials. From now on, let $C^{\delta}_k$ denote the usual ultraspherical polynomial (\cite[(4.7.1)
in p. 80]{Szego1975}).

\begin{prop}\label{prop 5.2}
For $n\geqslant0$ and $d\geqslant 3$, we get
\begin{eqnarray*}
\tilde{\mathbb{L}}_n(x,y) &=& \mathbb{L}_n(x,y) \\
& & -\frac{A_0}{\lambda} \,\sum_{k=0}^{n}
\mathbf{K}^{\mu,k+\delta}_{\left[\frac{n-k}{2}\right]}(2r^2-1,1)^t \,
\Lambda^{(\mu,k+\delta;M_k)}_{\left[\frac{n-k}{2}\right]}\, \mathbf{K}^{\mu,k+\delta}_{\left[\frac{n-k}{2}
\right]}(2s^2-1,1) \\
& & \quad\times \,2^{k}\, (r\,s)^{k}\,
\frac{k+\delta}{\delta} \,C_{k}^{\delta}(\langle \xi, \varrho \rangle),
\end{eqnarray*}
where $x=r \,\xi$, $y=s \,\varrho$, $r=\|x\|$, $s= \|y\|$, $\xi, \varrho \in \mathbb{S}^{d-1}$.
\end{prop}

\bigskip

\begin{proof}
Using (\ref{eq:sovpol}), (\ref{tilde--norm}), and (\ref{KC1v}), we get
\begin{eqnarray*}
\lefteqn{\hspace{-0.5cm}Q^m_{j,\nu}(x) Q^m_{j,\nu}(y)(\tilde{H}^m_{j,\nu})^{-1} }\\
&=&
q_j^{(\mu,\beta_j^m)}(2\,r^2 -1) Y_{\nu}^{m-2j}(x) q_j^{(\mu,\beta_j^m)}(2\,s^2 -1) Y_{\nu}^{m-2j}(y) \\
&~&\times \,2^{\beta_j^m+\mu+2}\, \frac{\omega_{\mu}}{\sigma_d} \,  (\tilde{h}_j^{(\mu,\beta_j^m)})^{-1} \\
&=&\Big(P_j^{(\mu,\beta_j^m)}(2\,r^2 -1)P_j^{(\mu,\beta_j^m)}(2\,s^2 -1)\Big(h_j^{(\mu,\beta_j^m)}\Big)^{-1}
 \\
& & - \mathbf{K}^{\mu,\beta_j^m}_{j}(2r^2-1,1)^t \, \Lambda^{(\mu,\beta_j^m;M_{m-2j})}_j\,
\mathbf{K}^{\mu,\beta_j^m}_{j}(2s^2-1,1) \\
& & + \mathbf{K}^{\mu,\beta_j^m}_{j-1}(2r^2-1,1)^t \, \Lambda_{j-1}^{(\mu,\beta_j^m;M_{m-2j})}\,
\mathbf{K}^{\mu,\beta_j^m}_{j-1}(2s^2-1,1)\Big) \\
& & \quad \times \,2^{m-2j}\,\frac{A_0}{\lambda} \, Y_{\nu}^{m-2j}(x)\, Y_{\nu}^{m-2j}(y).
\end{eqnarray*}
Then, summing above expressions for $m$, $j$, and $\nu$ we obtain
\begin{eqnarray*}
\lefteqn{\tilde{\mathbb{L}}_n(x,y)
= \mathbb{L}_n(x,y)}\\
& & - \sum_{m=0}^n\sum_{j=0}^{[m/2]}\sum_{\nu =1}^{a_{m-2j}^d}  \mathbf{K}^{\mu,\beta_j^m}_{j}(2r^2-1,1)^t
\, \Lambda^{(\mu,\beta_j^m;M_{m-2j})}_j\, \mathbf{K}^{\mu,\beta_j^m}_{j}(2s^2-1,1) \\
& & \quad\times \,2^{m-2j}\frac{A_0}{\lambda}\, (r\,s)^{m-2j}\, Y_{\nu}^{m-2j}(\xi)
\, Y_{\nu}^{m-2j}(\varrho) \\
& & +\sum_{m=2}^n\sum_{j=1}^{[m/2]}\sum_{\nu =1}^{a_{m-2j}^d}  \mathbf{K}^{\mu,\beta_j^m}_{j-1}(2r^2-1,1)^t
\, \Lambda^{(\mu,\beta_j^m;M_{m-2j})}_{j-1}\, \mathbf{K}^{\mu,\beta_j^m}_{j-1}(2s^2-1,1) \\
& & \quad\times \,2^{m-2j}\frac{A_0}{\lambda}\, (r\,s)^{m-2j}\, Y_{\nu}^{m-2j}(\xi)
\, Y_{\nu}^{m-2j}(\varrho),
\end{eqnarray*}
where we have used that $\mathbf{K}^{\mu,\beta_j^m}_{-1}(2r^2-1,1)=0$. Taking into account
the addition formula of spherical harmonics for $d\geqslant 3$ (see \cite[p. 9]{DaiXu2013})
$$\sum_{\nu =1}^{a_{k}^d} Y_{\nu}^{k}(\xi)\, Y_{\nu}^{k}(\varrho)= \frac{k+\delta}{\delta} \,
C_{k}^{\delta}(\langle \xi, \varrho \rangle),$$
we deduce
\begin{eqnarray*}
\lefteqn{\tilde{\mathbb{L}}_n(x,y) = \mathbb{L}_n(x,y)}\\
& & - \sum_{m=0}^n\sum_{j=0}^{[m/2]} \mathbf{K}^{\mu,\beta_j^m}_{j}(2r^2-1,1)^t \,
\Lambda^{(\mu,\beta_j^m;M_{m-2j})}_j\, \mathbf{K}^{\mu,\beta_j^m}_{j}(2s^2-1,1) \\
& & \quad\times \,2^{m-2j}\frac{A_0}{\lambda}\,(r\,s)^{m-2j}\,
\frac{m-2j+\delta}{\delta} \,C_{m-2j}^{\delta}(\langle \xi, \varrho \rangle) \\
& & +\sum_{m=2}^n\sum_{j=1}^{[m/2]} \mathbf{K}^{\mu,\beta_j^m}_{j-1}(2r^2-1,1)^t \,
\Lambda^{(\mu,\beta_j^m;M_{m-2j})}_{j-1}\, \mathbf{K}^{\mu,\beta_j^m}_{j-1}(2s^2-1,1) \\
& & \quad\times \,2^{m-2j}\frac{A_0}{\lambda}\, (r\,s)^{m-2j} \,
\frac{m-2j+\delta}{\delta} \,C_{m-2j}^{\delta}(\langle \xi, \varrho \rangle).
\end{eqnarray*}
Therefore, since $\beta_{j+1}^{m+2} = \beta^m_j$, a change in the indexes in the last term
gives
$$\tilde{\mathbb{L}}_n(x,y) = \mathbb{L}_n(x,y) - F(n) - F(n-1),$$
where
\begin{eqnarray*}
F(n) &=&\frac{A_0}{\lambda}\, \sum_{j=0}^{[\frac{n}{2}]}
\mathbf{K}^{\mu,\beta_j^n}_{j}(2r^2-1,1)^t \,
\Lambda^{(\mu,\beta_j^n;M_{m-2j})}_j\, \mathbf{K}^{\mu,\beta_j^n}_{j}(2s^2-1,1) \\
& & \quad\times \,2^{n-2j}\, \frac{\omega_{\mu}}{\sigma_d} \,(r\,s)^{n-2j}\,
\frac{n-2j+\delta}{\delta} \,C_{n-2j}^{\delta}(\langle \xi, \varrho \rangle),
\end{eqnarray*}
for $n\geqslant 0$, and $F(-1)=0$.

Finally, taking the change of indexes $n-2j=k$ in both expressions $F(n)$ and $F(n-1)$,
and summing, we get
\begin{eqnarray*}
\lefteqn{F(n) + F(n-1)}\\
& = &\frac{A_0}{\lambda}\,\sum_{k=0}^{n}
\mathbf{K}^{\mu,k+\delta}_{[\frac{n-k}{2}]}(2r^2-1,1)^t \,
\Lambda^{(\mu,k+\delta;M_k)}_{[\frac{n-k}{2}]}\, \mathbf{K}^{\mu,k+\delta}_{[\frac{n-k}{2}
]}(2s^2-1,1) \\
& & \quad\times \,2^{k}\,(r\,s)^{k}\,
\frac{k+\delta}{\delta} \,C_{k}^{\delta}(\langle \xi, \varrho \rangle),
\end{eqnarray*}
and the result follows.

\end{proof}

\section{Asymptotics for Christoffel functions}
\setcounter{equation}{0}

For the boundary of the ball, we shall prove:

\begin{theorem}\label{th1}
Assume that $\mu \geqslant-\frac{1}{2}$. For $\left\Vert x\right\Vert =1$,
\begin{equation}\label{ass1}
\lim_{n\rightarrow \infty }\frac{\mathbb{L}_{n}( x,x) -
\tilde{\mathbb{L}}_{n}( x,x) }{n^{2\mu +d+1}}=\frac{2}{\Gamma \left( 2\mu+d+2\right) }
\frac{( \mu +1) ( \mu +3) }{( \mu+2) ^{2}}.
\end{equation}
Moreover,
$$
\lim_{n\rightarrow \infty }\frac{\tilde{\mathbb{L}}_{n}( x,x) }{n^{2\mu +d+1}}=\frac{2}{\Gamma
( 2\mu +d+2) ( \mu +2)^{2}}.
$$
\end{theorem}

\bigskip

We note that the restriction $\mu \geqslant-\frac{1}{2}$ arises because existing
asymptotics for Christoffel functions in the non--Sobolev case have only been
established for this range of $\mu$. Asymptotics for the interior of the ball
have been obtained as well.

\begin{theorem}\label{th2}

For $r=\left\Vert x\right\Vert < 1$, we have
\begin{eqnarray*}
0 &<& \mathbb{L}_{n}( x,x) -\tilde{\mathbb{L}}_{n}( x,x)\\
&\leqslant  &  Cn^{d-1}
\log n \Big( 2 (1-r^{2}) +\frac{4}{n^{2}}\Big)^{-\mu -\frac{1}{2}}
\Big( 2r^{2}+\frac{4}{n^{2}}\Big) ^{-\delta -\frac{1}{2}}.
\end{eqnarray*}
Here $C$ is independent of $n$ and $x$. Consequently if $\mu \geqslant-\frac{1}{2}$, uniformly for $x$ in compact
subsets of $\left\{ x:0<\left\Vert x\right\Vert <1\right\} $,
\begin{equation}\label{ass2}
\lim_{n\rightarrow \infty }\tilde{\mathbb{L}}_{n}( x,x) / \binom{n+d}{d}
=\frac{1}{\sqrt{\pi }}
\frac{\Gamma ( \mu +1) \Gamma ( \frac{d+1}{2}) }{\Gamma ( \mu +\frac{d}{2}+1) }
\left(1-\left\Vert x\right\Vert ^{2}\right) ^{-\frac{1}{2}-\mu }.
\end{equation}
This last limit also holds for $x=0$.
\end{theorem}

\bigskip

In this section, we shall use the abbreviation that for $m=[ \frac{n-k}{2}] $,
\[
K_{m}( x,y) =K_{m}(x,y;\mu, k+\delta).
\]
Thus $k,m$ and $n$ are linked. We now turn to the Christoffel function.

\begin{lemma}\label{psin}
For $d\geqslant 3$ and $n\geqslant 0$, we get
$$
\tilde{\mathbb{L}}_{n}( x,x) =\mathbb{L}_{n}( x,x)-\Psi _{n}(x) ,
$$
where
\begin{equation}\label{psi}
\Psi _{n}(x) =\frac{A_{0}^{2}}{\lambda\,\delta }\sum_{k=0}^{n}2^{2k}( k+\delta)
\binom{k+d-3}{k} r^{2k} F_{k,m}(t).
\end{equation}
Here \,\,$x=r\xi , \,\,r=\left\Vert x\right\Vert, \,\,m =
[ \frac{n-k}{2}], \,\,t=2r^{2}-1,$ and
\begin{equation}\label{L3-Fkm}
F_{k,m}(t) =\frac{k^{2} K_{m}( t,1)^{2}+8 k K_{m}(t,1) K_{m}^{( 0,1)}( t,1) +
16 K_{m}^{(0,1)}(t,1)^{2}}{1+2^{k} A_{0} \left\{ k^{2} K_{m}(1,1) + 8 k K_{m}^{(1,0) }
(1,1) + 16 K_{m}^{(1,1) }(1,1) \right\} }.
\end{equation}
\end{lemma}

\bigskip

\begin{proof}
From Proposition \ref{prop 5.2},
\begin{eqnarray}
\Psi _{n}(x) &=&\frac{A_0}{\lambda}\sum_{k=0}^{n}
\mathbf{K}_{[\frac{n-k}{2}]}^{\mu ,k+\delta}(2r^{2}-1,1)^{t}
\Lambda _{[ \frac{n-k}{2}] }^{(\mu ,k+\delta;M_{k})}
\mathbf{K}_{[ \frac{n-k}{2}] }^{\mu,k+\delta}(2r^{2}-1,1)
\nonumber
\\\label{aux1-psi}
&&\times 2^{k}r^{2k}\frac{k+\delta }{\delta }C_{k}^{\delta }(1) .
\end{eqnarray}
Here $C_{k}^{\delta }$ is an ultraspherical polynomial, so that \cite[p. 80, (4.7.3)]{Szego1975}
\begin{equation}\label{cheby1}
C_{k}^{\delta }\left( 1\right) = \binom{k+2\delta -1}{k} = \binom{k+d-3}{k}.
\end{equation}
Using Lemma~\ref{le2} and (\ref{Kalpha}), a straightforward computation shows that
\begin{eqnarray*}
\lefteqn{\mathbf{K}_{[ \frac{n-k}{2}] }^{\mu ,k+\delta}( 2r^{2}-1,1) ^{t}
\Lambda _{[ \frac{n-k}{2}]}^{( \mu ,k+\delta;M_{k}) }
\mathbf{K}_{[ \frac{n-k}{2}] }^{\mu ,k+\delta}(2r^{2}-1,1) }
\\
&=&\frac{2^{k}A_{0}}{\Delta _{k,m}}\left\{ k^{2}K_{m}(t,1)^{2} + 8 k K_{m}(t,1)
K_{m}^{(0,1)}(t,1) + 16 K_{m}^{(0,1)}(t,1)^{2}\right\} .
\end{eqnarray*}
Substituting this, (\ref{delta}) and (\ref{cheby1}) into (\ref{aux1-psi}), gives the result.

\end{proof}

\bigskip

In particular, for $r=\|x\|=1$, we see that $t=1$ and
$$
\Psi _{n}(x) =\frac{A_{0}^{2}}{\lambda\,\delta }\sum_{k=0}^{n}2^{2k}(k+\delta)
\binom{k+d-3}{k} F_{k,m} \left( 1\right),
$$
where
\begin{equation}\label{fkm1}
F_{k,m}(1) =\frac{k^{2} K_{m}(1,1)^{2} + 8 k K_{m}(1,1)
K_{m}^{(0,1)}(1,1) + 16 K_{m}^{(0,1) } (1,1)^{2}}{1+2^{k} A_{0}
\left\{ k^{2} K_{m}(1,1) + 8 k K_{m}^{(1,0)}(1,1) + 16 K_{m}^{(1,1)}(1,1) \right\} }.
\end{equation}

\bigskip

Next, we obtain asymptotics involving the reproducing kernel as $m\rightarrow \infty:$

\begin{lemma}\label{le6.4}
As $m\rightarrow \infty ,$ uniformly for $k\geqslant0,$ the following asymptotics hold

\begin{itemize}
\item[\hbox{{\rm (i)}}]
\begin{equation}\label{L4a}
K_{m}(1,1) =2^{-k} (m+k)^{\mu +1} m^{\mu+1} B_{0} (1+o(1)) ,
\end{equation}
where
\begin{equation}\label{b0}
B_{0}=\frac{2^{-\mu -\delta -1}}{\Gamma(\mu +1) \Gamma (
\mu +2) }.
\end{equation}

\item[\hbox{{\rm (ii)}}]
\begin{equation}\label{L4b}
\hspace{-6pt}\frac{K_{m}^{(0,1) }(1,1) }{K_{m}(1,1)}=
\frac{(m+\mu +k+\delta +2) m}{2 (\mu +2) }=\frac{(m+k) m}{2
( \mu +2) }(1+o(1)).
\end{equation}

\item[\hbox{{\rm (iii)}}]
\begin{eqnarray}
\frac{K_{m}^{( 1,1) }( 1,1) }{K_{m}( 1,1) }&=&
\frac{( m+\mu +k+\delta +2) m}{4( \mu +1) (\mu +2) ( \mu +3) }
\nonumber\\
&~& \times\,(( \mu +2) m(m+\mu +k+\delta +2) +k+\delta)
\nonumber
\\ \label{L4c}
&=&\frac{( m+k) ^{2}m^{2}}{4( \mu +1)
( \mu +3) }(1+o(1)) .
\end{eqnarray}
\end{itemize}
\end{lemma}

\begin{proof}
From Lemma \ref{Kn} with $\alpha =\mu $ and $\beta =k+\delta $, the formulas follow using $\Gamma ( x+1) =x\Gamma ( x)$
and the following consequence of Stirling's formula: for fixed $a,b$, as $x\rightarrow \infty$,
\[
\frac{\Gamma (x+b)}{\Gamma ( x+a) }=x^{b-a}( 1+o(
1) ).
\]

\end{proof}

\bigskip

\begin{lemma}\label{le5}

\begin{itemize}
\item[\hbox{{\rm (i)}}]\label{le5a}
Let
\begin{equation}\label{d0}
D_{0}=B_{0}\frac{( \mu +1) ( \mu +3) }{A_{0}( \mu+2) ^{2}}.
\end{equation}
Then as $m\rightarrow \infty $, uniformly in $k\geqslant0,$
$$
2^{2k}F_{k,m}( 1) =D_{0}( m+k) ^{\mu +1}m^{\mu+1}
(1+o(1)).
$$

\item[\hbox{{\rm (ii)}}]\label{le5b} For $\|x\| =1$, let
$$
\Psi _{n,1}(x) =\frac{A_{0}^{2}}{\lambda\,\delta }\sum_{k=0}^{n-\left[\log n\right] }2^{2k}
( k+\delta ) \binom{k+d-3}{k} F_{k,m} (1).
$$
Then
\begin{equation}\label{e0}
\lim_{n\rightarrow \infty }\frac{\Psi _{n,1}(x) }{n^{2\mu +d+1}}=E_{0},
\end{equation}
where
\begin{equation}\label{E0}
E_{0}=\frac{2}{\Gamma ( 2\mu +d+2) }\frac{( \mu +1)
( \mu +3) }{( \mu +2)^{2}}.
\end{equation}
\end{itemize}
\end{lemma}

\bigskip

\begin{proof}

(i) Assuming $n-k\rightarrow \infty$, so that $m=[ \frac{n-k}{2}]\rightarrow \infty $,
the previous lemma gives for the term in the numerator in (\ref{fkm1}),
\begin{eqnarray*}
\lefteqn{k^{2}K_{m}( 1,1) ^{2}+8kK_{m}( 1,1) K_{m}^{(0,1) }( 1,1)
 +16K_{m}^{( 0,1) }( 1,1)^{2} }
\\
&=&\hspace{-5pt}K_{m}(1,1)^{2}\,\\
&\times&\hspace{-5pt}\Big\{ k^{2} + 8 k \frac{(m+k) m}{2(\mu +2) }
(1+o(1)) + 4\frac{(m+k)^2 m^2}{(\mu +2)^2
}(1+o(1)) \Big\}
\\
&=&\hspace{-5pt}K_{m}( 1,1) ^{2}4\frac{(m+k)^2 m^2}{(\mu +2)^2
}(1+o(1)).
\end{eqnarray*}
Also the term in the denominator in (\ref{fkm1}) has the form
\begin{eqnarray*}
\lefteqn{\hspace*{-20pt}1+2^{k}A_{0}\Big\{ k^{2}K_{m}( 1,1) +8kK_{m}^{(1,0) }( 1,1) +
16K_{m}^{( 1,1) }( 1,1)\Big\} }
\\
&=&1+2^{k}A_{0}K_{m}( 1,1)\,\Big\{ k^{2} + 4 k \frac{(m+k) m}{( \mu +2) }
(1+o(1))\\
&~& + 4 \frac{( m+k)^{2}m^{2}}{( \mu +1) ( \mu +3) }
(1+o(1)) \Big\}
\\
&=&1+2^{k+2}A_{0}K_{m}( 1,1) \frac{( m+k)^{2}m^{2}}{( \mu +1)
( \mu +3) }(1+o(1)) .
\end{eqnarray*}
Thus
\[
F_{k,m}( 1) =\frac{K_{m}( 1,1)^{2} 4 \frac{(m+k)^2 m^2}{(\mu +2)^2}
(1+o(1)) }{1+2^{k+2}A_{0}K_{m}( 1,1)
\frac{( m+k)^{2}m^{2}}{( \mu +1) ( \mu +3) }
(1+o(1))}.
\]
Here
\[
2^{k}K_{m}( 1,1) =( m+k) ^{\mu +1}m^{\mu+1}B_{0}( 1+o( 1) )
\rightarrow \infty \textrm{ as } m\rightarrow \infty ,
\]
so
\begin{eqnarray*}
2^{2k}F_{k,m}( 1) &=&2^{k}K_{m}( 1,1) \frac{( \mu+1)
( \mu +3) }{A_{0}( \mu +2) ^{2}}(1+o(1))
\\
&=&D_{0}( m+k) ^{\mu +1}m^{\mu +1}(1+o(1)) .
\end{eqnarray*}

\noindent
(ii) From (i), and as $\frac{m}{n}=\frac{1}{2}( 1-\frac{k}{n})
+O( \frac{1}{n}) ,$
\begin{eqnarray*}
\lefteqn{\frac{\Psi _{n,1}( x) }{n^{2\mu +2}} = \frac{A_{0}^{2}D_{0}}{\lambda\,\delta }}\\
\\
&&\times\sum_{k=0}^{n-\left[ \log n\right] }(k+\delta )
\binom{k+d-3}{k}\Big( \frac{m+k}{n}\Big)^{\mu +1}
\Big(\frac{m}{n}\Big)^{\mu +1}(1+o(1))
\\
&=&\frac{A_{0}^{2} D_{0} n^{d-2}}{\lambda\,\delta 2^{2\mu +2} ( d-3)!}(1+o( 1))\\
&~&\times\sum_{k=0}^{n-[ \log n] }\Big(
\Big( \frac{k}{n}\Big) ^{d-2}\Big(1-\Big( \frac{k}{n}\Big)^{2}\Big)^{\mu +1}+
O\Big( \frac{1}{n}\Big) \Big)
\\
&=&\frac{A_{0}^{2}D_{0}n^{d-1}}{\lambda\,\delta 2^{2\mu +2}\left( d-3\right) !}
(1+o( 1) )\Big( \int_{0}^{1}x^{d-2}\left( 1-x^{2}\right) ^{\mu+1}dx+o\left( 1\right) \Big).
\end{eqnarray*}
Here, setting $x=s^{1/2}$,
$$
\int_{0}^{1}x^{d-2}\left( 1-x^{2}\right) ^{\mu +1}dx=\frac{1}{2}\int_{0}^{1}s^{d/2-3/2}
\left( 1-s\right) ^{\mu +1}ds=\frac{\Gamma ( \frac{d-1}{2})
\Gamma ( \mu +2) }{2\Gamma ( \frac{d+2\mu +3}{2}) }.
$$
Then
$$
\frac{\Psi _{n,1}( x) }{n^{2\mu +d+1}}=\frac{A_{0}^{2}D_{0}}{\lambda\,\delta\,
2^{2\mu +2}( d-3) !}\frac{\Gamma ( \frac{d-1}{2}) \Gamma
 ( \mu +2) }{2\Gamma ( \frac{d+2\mu +3}{2}) }( 1+o( 1) ) .
$$
Now we simplify the constant. Using (\ref{omegamu}), (\ref{sigmad}), (\ref{a0}), (\ref{b0}), and
(\ref{d0}),
\begin{eqnarray*}
\lefteqn{\hspace{-15pt}\frac{A_{0}^{2}D_{0}}{\lambda\,\delta ( d-3) !\,2^{2\mu +2}}\frac{\Gamma
( \frac{d-1}{2}) \Gamma ( \mu +2) }{2\Gamma (
\frac{d+2\mu +3}{2}) } }\\
&=&\frac{1}{\lambda\,( d-2) !\,2^{2\mu +1}}A_{0}B_{0}\frac{( \mu
+1) ( \mu +3) }{( \mu +2) ^{2}}\frac{\Gamma
( \frac{d-1}{2}) \Gamma ( \mu +2) }{2\Gamma (
\frac{d+2\mu +3}{2}) }
\\
&=&\frac{1}{\left( d-2\right) !\,2^{2\mu +2}}\frac{\Gamma ( \frac{d}{2}
) }{\Gamma ( \mu +\frac{d}{2}+1) }\frac{( \mu
+1) ( \mu +3) }{( \mu +2) ^{2}}\frac{\Gamma
( \frac{d-1}{2}) }{\Gamma ( \frac{d+2\mu +3}{2}) }
=E_{0},
\end{eqnarray*}
say. Using Legendre's duplication formula
\[
\Gamma ( 2a) =\frac{2^{2a-1}}{\sqrt{\pi }}\Gamma ( a)
\Gamma \Big( a+\frac{1}{2}\Big) ,
\]
with $a=\mu +\frac{d}{2}+1$, we see that
\[
\Gamma \Big( \mu +\frac{d}{2}+1\Big) \, \Gamma \Big( \mu +\frac{d}{2}+\frac{%
3}{2}\Big) =2^{-2\mu -d-1}  \sqrt{\pi }\, \Gamma ( 2\mu +d+2),
\]%
and with $a=\frac{d-1}{2}$,
\[
\Gamma \Big( \frac{d-1}{2}\Big)\, \Gamma \Big( \frac{d}{2}\Big) = 2^{-d+2}\sqrt{\pi }\,\Gamma( d-1).
\]%
So, finally we obtain
$$
E_{0} =\frac{2}{\Gamma ( 2\mu +d+2) }\frac{( \mu +1)
( \mu +3) }{( \mu +2) ^{2}}.
$$

\end{proof}

\bigskip

We shall need an estimate on the reproducing kernels that is uniform in $k$:

\begin{lemma}\label{le6}

Fix $\mu >-1,\delta \geqslant0$. For $m=\left[ \frac{n-k}{2}\right] \geqslant1,\,k\geqslant
0$, and $t\in \left[ -1,1\right],$
\begin{eqnarray*}
\lefteqn{K_{m}( t,t)  = K_{m}(
t,t;\mu,k+\delta)   \nonumber }\\
&\leqslant  &C( 1+t) ^{-k}\Big( m+\Big[ \frac{k}{2}\Big] +1\Big)
\\
&~& \times \Big( 1-t+\frac{1}{( m+[ \frac{k}{2}] +1) ^{2}}\Big)
^{-\mu -\frac{1}{2}}\Big( 1+t+\frac{1}{( m+[ \frac{k}{2}]+1
) ^{2}}\Big) ^{-\delta -\frac{1}{2}}.
\end{eqnarray*}
Here $C$ depends on $\mu $ and $\delta $ but not on $k,n,t$.
\end{lemma}

\bigskip

\begin{proof}

Suppose first $k$ is even, say $k=2\ell $. Then from the extremal properties
for Christoffel functions,
\begin{eqnarray*}
\lefteqn{K_{m}( t,t;\mu,k+\delta)  = \sup_{\deg (
P) \leqslant  m}\frac{P^{2}( t) }{\int_{-1}^{1}P^{2}(
s) \left( 1-s\right) ^{\mu }\left( 1+s\right) ^{k+\delta }ds}}
\\
&& =( 1+t) ^{-k}\sup_{\deg ( P) \leqslant  m}\frac{\left(
P( t) \left( 1+t\right) ^{\ell }\right)^{2}}{\int_{-1}^{1}\left(
P( s) \left( 1+s\right) ^{\ell }\right)^{2}\left( 1-s\right)
^{\mu }\left( 1+s\right) ^{\delta }ds}
\\
&&\leqslant  ( 1+t) ^{-k}\sup_{\deg ( R) \leqslant  m+\ell}
\frac{R( t) ^{2}}{\int_{-1}^{1}R( s) ^{2}\left(
1-s\right) ^{\mu }\left( 1+s\right) ^{\delta }ds} \\
&& = ( 1+t) ^{-k}K_{m+\left[ \frac{k}{2}\right] }( t,t;\mu,\delta) .
\end{eqnarray*}
We now use a result from Nevai's 1979 Memoir \cite[p. 108, Lemma 5]
{Nevai1979}, that for $m+[ \frac{k}{2}] \geqslant1$ and $t\in \left[-1,1\right]$,
\begin{eqnarray*}
\lefteqn{K_{m+[ \frac{k}{2}] }(
t,t;\mu,\delta) \leqslant  C\Big( m+\left[ \frac{k}{2}\right] +1\Big)}
\\
&&\times  \Big( 1-t+\frac{1}{( m+[ \frac{k}{2}] +1) ^{2}}\Big)^{-\mu -\frac{1}{2
}}\Big( 1+t+\frac{1}{( m+[ \frac{k}{2}] +1)^{2}}\Big)^{-\delta-\frac{1}{2}}.
\end{eqnarray*}
The case $k=2\ell +1$ is similar.

\end{proof}

\bigskip

Now, we have the necessary tools in order to prove the main theorems of this section.

\begin{proof}[Proof of Theorem \ref{th1}]
We already have a limit for $\Psi _{n,1}$, and must now estimate the
remaining part of $\Psi _{n}$, namely, for $\|x\|=1$,
\[
\Psi _{n,2}(x) =\frac{A_{0}^{2}}{\lambda\,\delta }\sum_{k=n-
\left[ \log n\right] +1}^{n}2^{2k}( k+\delta ) \binom{k+d-3}{k} F_{k,m} (1) .
\]
We shall show that $\Psi _{n,2}(x) =o( n^{2\mu +d+1})$ which,
together with (\ref{e0}), will give the result. Now if $m\geqslant1$,
\begin{eqnarray*}
&&\hspace*{-20pt}{2^{k}A_{0}F_{k,m}( 1)
  =2^{k}A_{0}\frac{k^{2}K_{m}(
1,1) ^{2}+8kK_{m}( 1,1) K_{m}^{( 0,1) }(
1,1) +16K_{m}^{( 0,1) }( 1,1) ^{2}}{
1+2^{k}A_{0}\left\{ k^{2}K_{m}( 1,1) +8kK_{m}^{( 1,0)
}( 1,1) +16K_{m}^{( 1,1) }( 1,1) \right\} }}
\\
&&\leqslant  K_{m}( 1,1) +K_{m}( 1,1) +\frac{K_{m}^{(
0,1) }( 1,1) ^{2}}{K_{m}^{( 1,1) }(
1,1) }.
\end{eqnarray*}
Here as $k$ is close to $n$, and $m=O( \log n)$, (\ref{L4b}) and (\ref{L4c})
give
\[
\Big( \frac{K_{m}^{( 0,1) }( 1,1) }{K_{m}(1,1) }
\Big) ^{2}\leqslant  C( nm) ^{2},
\]
while if $m\geqslant1$,
\[
\frac{K_{m}( 1,1) }{K_{m}^{( 1,1) }( 1,1) }\leqslant  C( n^{2}m^{2}) ^{-1}.
\]
When $m=0$, the estimation is simpler as $K_{m}^{( 0,1)
}=0=K_{m}^{( 1,1) }$. Thus
\[
2^{k}A_{0}F_{k,m}( 1) \leqslant  CK_{m}( 1,1) ,
\]
where $C$ is a constant independent of $m$ and $n$, so for some possibly
different $C$,
\[
\Psi _{n,2}\left( x\right) \leqslant  C\sum_{k=n-\left[ \log n\right]+1}^{n}2^{k}( k+\delta )
\binom{ k+d-3}{k} K_{m} ( 1,1).
\]
Here, using Lemma \ref{le6.4},
\[
2^{k}K_{m}( 1,1) \leqslant  Cn^{2\mu +2},
\]
while
\[
( k+\delta ) \binom{ k+d-3}{k}\leqslant  Cn^{d-2}.
\]
Thus
\[
\Psi _{n,2}( x) \leqslant  C\, n^{2\mu +d}\,\log n=o( n^{2\mu+d+1}) .
\]

Then (\ref{ass1}) follows from Lemma~\ref{le5}. Finally, we note that if $\mu \geqslant-\frac{1}{2}$,
(1.10) of Theorem 1.1 in \cite[p. 120]{KrooLubinsky2013B} gives
\[
\lim_{n\rightarrow \infty }\mathbb{L}_{n}( x,x) /n^{2\mu +d+1}=\frac{2}{\Gamma
( 2\mu +d+2) }.
\]
Take there $\rho =\mu +\frac{1}{2}$, and note that the normalization
constant $\omega _{\rho }$ is incorporated in \cite[p. 119]{KrooLubinsky2013B} in
a different way to that here. Then
\begin{eqnarray*}
\lim_{n\rightarrow \infty }\tilde{\mathbb{L}}_{n}( x,x) /n^{2\mu
+d+1} &=&\frac{2}{\Gamma ( 2\mu +d+2) }-\frac{2}{\Gamma (
2\mu +d+2) }\frac{( \mu +1) ( \mu +3) }{(
\mu +2) ^{2}}
\\
&=&\frac{2}{\Gamma ( 2\mu +d+2) }\left\{ 1-\frac{( \mu
+1) ( \mu +3) }{( \mu +2) ^{2}}\right\}
\\
&=&\frac{2}{\Gamma ( 2\mu +d+2) ( \mu +2) ^{2}}.
\end{eqnarray*}

\end{proof}

\bigskip

Next we deal with $\left\Vert x\right\Vert <1$.

\begin{proof}[Proof of Theorem \ref{th2}]

We must estimate $F_{k,m}( t)$ defined in (\ref{L3-Fkm}), with $t=2r^{2}-1$ and $r=\|x\|$.

Let us assume that $t\leqslant  1-\eta$ for some $\eta >0$. Then, with the
convention $p_{j}=p_{j}^{( \mu ,k+\delta ) }$ for orthonormal Jacobi polynomials,
$j=m,m+1,$
\begin{eqnarray}
\left\vert K_{m}( t,1) \right\vert  &=&\frac{\gamma _{m}}{
\gamma _{m+1}}\left\vert \frac{p_{m+1}( t) p_{m}( 1)
-p_{m}( t) p_{m+1}( 1) }{t-1}\right\vert
\nonumber
\\
&\leqslant  &\frac{C}{2}\frac{\sqrt{p_{m}^{2}( t) +p_{m+1}^{2}(
t) }\sqrt{p_{m}^{2}( 1) +p_{m+1}^{2}( 1) }}{\eta
}
\nonumber
\\ \label{0cotakm}
&\leqslant  &\frac{C}{2\eta }K_{m+1}( t,t) ^{1/2}\sqrt{%
p_{m}^{2}( 1) +p_{m+1}^{2}( 1) },
\end{eqnarray}%
where $\gamma_m = k_m/\sqrt{h_m}$ is the leading coefficient of $p_{m}$.

Also,
\begin{eqnarray}
\lefteqn{\left\vert K_{m}^{( 0,1) }( t,1) \right\vert =\frac{\gamma _{m}}{\gamma _{m+1}} }
\nonumber
\\
&\times& \hspace{-8pt}\left| \frac{( p_{m+1}(t)
p_{m}'(1)-p_{m}(t) p_{m+1}'(1))(t-1)+p_{m+1}(t)
p_{m}(1)-p_{m}(t)p_{m+1}(1)}{(t-1)^{2}}\right|
\nonumber
\\
&\leqslant  &\hspace{-8pt}\frac{C}{2\eta ^{2}}\sqrt{p_{m}^{2}(t) +p_{m+1}^{2}(
t)}\left\{2\sqrt{( p_{m}'(1))^{2}+(p_{m+1}'(1))^{2}}+\sqrt{p_{m}^{2}(1)+p_{m+1}^{2}(1) }\right\}
\nonumber
\\\label{0cotakm01}
&\leqslant  &\hspace{-10pt}\frac{C}{2\eta ^{2}}K_{m+1}(t,t)^{1/2}\left\{ 2\sqrt{(p_{m}'(1))^{2}+
(p_{m+1}'(1))^{2}}+\sqrt{p_{m}^{2}(1)+p_{m+1}^{2}(1) }\right\}.
\end{eqnarray}
Next, we note that given any real number $a$, there exists $C_{a}>1$ such
that for all $x$ with $\min ( x,x+a) \geqslant1,$
\[
C_{a}^{-1}x^{a}\leqslant  \frac{\Gamma ( x+a) }{\Gamma ( x) }\leqslant  C_{a}x^{a}.
\]
This follows from Stirling's formula and the positivity and continuity of
$\frac{\Gamma ( x+a) }{\Gamma ( x) }$ for this range of $x$. Then
from (\ref{JacON2}), if $m\geqslant1,$ $\alpha =\mu$, $\beta= k+\delta$,
\begin{equation}\label{pm'1}
\left\vert p_{m}^{\prime }( 1) \right\vert
\leqslant  C\frac{( m+k) ^{3/2+\mu /2}\,m^{1+\mu /2}}{2^{k/2}},
\end{equation}
and from (\ref{JacON1}),
$$
\left\vert p_{m}( 1) \right\vert
\leqslant   C\frac{( m+k) ^{1/2+\mu /2}\,m^{\mu /2}}{2^{k/2}}.
$$
Substituting these into (\ref{0cotakm}) and (\ref{0cotakm01}) gives for $m\geqslant1,$
\begin{equation}\label{cotaKm}
\left\vert K_{m}( t,1) \right\vert \leqslant  C
\Big( \frac{K_{m+1}( t,t) }{2^{k}}\Big) ^{1/2}( m+k) ^{1/2+\mu/2}\,m^{\mu /2}
\end{equation}
and
\begin{equation}\label{cotaKm01}
\left\vert K_{m}^{( 0,1) }( t,1) \right\vert
\leqslant  C\Big( \frac{K_{m+1}( t,t) }{2^{k}}\Big) ^{1/2}(m+k)^{3/2+\mu /2}\,m^{1+\mu /2}.
\end{equation}
Next, by (\ref{L4a}) and (\ref{L4c}),
\begin{equation}\label{cotaKm11}
2^{k}K_{m}^{( 1,1) }( 1,1) \geqslant C( m+k)^{\mu +3}\,m^{\mu +3},
\end{equation}
so, inserting (\ref{cotaKm}), (\ref{cotaKm01}) and (\ref{cotaKm11})  into (\ref{L3-Fkm}),
\begin{eqnarray*}
\lefteqn{2^{k}F_{k,m}( t)  \leqslant  CK_{m+1}( t,t) }\\
&~&\times\left\{ \frac{k^{2}\,( m+k) ^{1+\mu }\,m^{\mu }+k\,( m+k) ^{2+\mu}\,
m^{1+\mu }+( m+k) ^{3+\mu }\,m^{2+\mu }}{( m+k) ^{\mu
+3}\,m^{\mu +3}}\right\}
\nonumber
\\\label{cotafkm}
&\leqslant  &CK_{m+1}\left( t,t\right) /\left( m+1\right) .
\end{eqnarray*}

This bound holds also for $m=0$, thought it is obtained in a simpler way since $K_{0}$ is a constant,
\[
F_{k,0}( t) =\frac{k^{2}K_{0}( t,1) ^{2}}{1+2^{k}A_{0}k^{2}K_{0}
( 1,1) }\leqslant  C\frac{K_{0}( t,t) }{2^{k}} \leqslant  C\frac{K_{1}( t,t) }{2^{k}} .
\]
Then,
\[
\Psi _{n}( x) \leqslant  C\sum_{k=0}^{n}2^{k}( k+\delta )
\binom{ k+d-3}{k} r^{2k} \frac{K_{m+1}( t,t) }{m+1}.
\]
Using Lemma~\ref{le6}, and that $1+t=2r^{2}$ and $m=\left[\frac{n-k}{2}\right]$, we continue this as
\begin{eqnarray*}
\lefteqn{\Psi _{n}( x) }
\\
&\leqslant  &C \sum_{k=0}^{n}( k+1)^{d-2}
\Big( \frac{m+[ \frac{k}{2}]+2 }{[ \frac{n-k}{2}]
+1}\Big)
\\
&&\times \Big( 1-t+\frac{1}{(m+[ \frac{k}{2}] +2)^{2}}\Big)^{-\mu -\frac{1}{2}}
\Big( 1+t+\frac{1}{( n+[ \frac{k}{2}] +2) ^{2}}\Big)^{-\delta -\frac{1}{2}}
\\
&\leqslant  &C n^{d-1}\Big( 1-t+\frac{4}{n^{2}}\Big) ^{-\mu -\frac{1}{2}}
\Big( 1+t+\frac{4}{n^{2}}\Big)^{-\delta -\frac{1}{2}}\sum_{k=0}^{n}\frac{1}{[
\frac{n-k}{2}] +1}
\\
&\leqslant  &C n^{d-1}\log n \Big( 2(1-r^{2}) +\frac{4}{n^{2}}\Big)^{-\mu -\frac{1}{2}}
\Big( 2r^{2}+\frac{4}{n^{2}}\Big) ^{-\delta -\frac{1}{2}}.
\end{eqnarray*}
Finally,  \cite[Theorem 1.3]{KrooLubinsky2013A} gives
\begin{eqnarray*}
\lim_{n\rightarrow \infty }\mathbb{L}_{n}\left( x,x\right) /\binom{n+d}{d}
&=& \frac{\omega _{\mu }W_{0}( x) }{\left( 1-\left\Vert
x\right\Vert ^{2}\right) ^{\mu }}\\
&=&\frac{1}{\sqrt{\pi }}
\frac{\Gamma (\mu +1)
\Gamma ( \frac{d+1}{2}) }{\Gamma ( \mu +\frac{d}{2}+1) }
\left( 1-\left\Vert x\right\Vert ^{2}\right) ^{-\frac{1}{2}-\mu },
\end{eqnarray*}
uniformly for $x$ in compact subsets of the unit ball. Thus $\mathbb{L}_{n}( x,x) $
grows like $n^{d}>>n^{d-1}\log n$, so (\ref{ass2}) follows.

\bigskip

It remains to deal with the case $x=0$, that is $r=0$. In this case all
terms in $\Psi _{n}( x) $ in (\ref{psi}) vanish except for $k=0$. We see
that
\begin{equation}\label{psi0}
\Psi _{n}( 0) =A_{0}^{2}F_{0,[ \frac{n}{2}] }(-1) =
\frac{A_{0}^{2}16K_{[ \frac{n}{2}] }^{(0,1) }( -1,1) ^{2}}{1+16A_{0}
K_{[ \frac{n}{2}]}^{( 1,1) }( 1,1) }.
\end{equation}
With $m=[ \frac{n}{2}] $, $k=0$, we see as above that
\begin{eqnarray}
\left\vert K_{m}^{( 0,1) }( -1,1) \right\vert &\leqslant &
C \Big(
\left\vert p_{m+1}( -1) \right\vert \left\vert p_{m}^{\prime
}( 1) \right\vert +\left\vert p_{m}( -1) \right\vert
\left\vert p_{m+1}^{\prime }( 1) \right\vert  \nonumber\\
&~&\quad +\left\vert p_{m+1}( -1) \right\vert \left\vert p_{m}(
1) \right\vert +\left\vert p_{m}( -1) \right\vert
\left\vert p_{m+1}( 1) \right\vert\Big).\label{cotakm01}
\end{eqnarray}

We shall need the classic bound \cite[p. 36, eqn. (20--21)]{NevaiVertesi1985}
\[
\left\vert p_{m}( t) \right\vert \leqslant  C
\Big( 1-t+\frac{1}{m^{2}}\Big) ^{-\frac{\mu }{2}-\frac{1}{4}}
\Big( 1+t+\frac{1}{m^{2}}\Big) ^{-\frac{\delta }{2}-\frac{1}{4}},\quad t\in [ -1,1].
\]
Here $C$ depends only on $\mu $ and $\delta $. Then
\[
\left\vert p_{m}( -1) \right\vert \leqslant  Cm^{\delta +\frac{1}{2}},
\qquad \left\vert p_{m}( 1) \right\vert \leqslant  Cm^{\mu +\frac{1}{2}}.
\]
Moreover, (\ref{pm'1}) gives (recall $k=0$),
\[
\left\vert p_{m}^{\prime }( 1) \right\vert \leqslant  Cm^{\mu +\frac{5}{2}}.
\]
Substituting all these bounds in (\ref{cotakm01}) yields
\[
\left\vert K_{m}^{( 0,1) }( -1,1) \right\vert \leqslant  Cm^{\delta +\mu +3}.
\]
In addition, (\ref{cotaKm11}) leads to
\[
\left\vert K_{m}^{( 1,1) }( 1,1) \right\vert \geqslant Cm^{2\mu +6}.
\]
Substituting the last two bounds in (\ref{psi0}) gives
\[
\left\vert \Psi _{n}\left( 0\right) \right\vert \leqslant  Cn^{2\delta}=Cn^{d-2}.
\]
Then (\ref{ass2}) follows also for this case.

\end{proof}

\bigskip

\section{The two dimensional case}
\setcounter{equation}{0}

In the case $d=2$ results are somewhat different, but Theorems \ref{th1} and \ref{th2} also hold. In this case
$\delta=0$,  $\omega _{\mu }=\pi$, $\sigma _{d}=2\pi$, then $\omega _{\mu }/\sigma _{d}=1/2$ and
$A_0=\lambda \,2^{\mu +1}$.
Moreover the reproducing kernel of spherical harmonics is obtained in a different way.

\begin{prop}\label{Kerneld2}
For $n\geqslant0$ and $d=2$, we get
\begin{eqnarray*}
\tilde{\mathbb{L}}_n(x,y) &=& \mathbb{L}_n(x,y) \\
& & -\frac{A_0}{\lambda} \,\sum_{k=0}^{n}
\mathbf{K}^{\mu,k}_{[\frac{n-k}{2}]}(2r^2-1,1)^t \,
\Lambda^{(\mu,k;M_k)}_{[\frac{n-k}{2}]}\, \mathbf{K}^{\mu,k}_{[\frac{n-k}{2}
]}(2s^2-1,1) \\
& & \quad\times \,2^{k}\, (r\,s)^{k}\,\cos(n(\theta-\widehat{\theta})),
\end{eqnarray*}
where $x=r \,(\cos \theta,\sin \theta)$, $y=s \,(\cos \widehat{\theta},\sin \widehat{\theta})$,
$r=\|x\|$, $s= \|y\|$, $\theta, \widehat{\theta} \in [0, 2\pi]$.
\end{prop}

\bigskip

\begin{proof}
The proof is the same as in Proposition \ref{prop 5.2} taking into account that in this case
$a_{k}^d=2$, \, for $k\geqslant0$, and the addition formula of spherical harmonics for $d=2$ reduces to
the addition formula for the cosines (see \cite[p. 20]{DaiXu2013}), then
$$\sum_{\nu =1}^{a_{k}^d} Y_{\nu}^{k}(\cos \theta,\sin \theta)\, Y_{\nu}^{k}
(\cos \widehat{\theta},\sin \widehat{\theta})= \cos(n(\theta-\widehat{\theta})). $$

\end{proof}

\bigskip

Lemma \ref{psin} can be rewritten for the case $d=2$ as

\begin{lemma}
For $d=2$ and $n\geqslant 0$, we get
$$
\tilde{\mathbb{L}}_{n}(x,x) =\mathbb{L}_{n}(x,x)-\Psi _{n}(x) ,
$$
where
$$
\Psi _{n}(x) =\frac{A_{0}^{2}}{\lambda}\sum_{k=0}^{n}2^{2k}r^{2k}F_{k,m}(t).
$$
Here \,$r=\left\Vert x\right\Vert $, \,$m=[ \frac{n-k}{2}]$,\,
$t=2r^{2}-1,$ \, and $F_{k,m}( t)$ is given as in (\ref{L3-Fkm}).
\end{lemma}

\bigskip

Lemma \ref{le6.4} and part (i) of Lemma \ref{le5a} are true for $d=2$, with $\delta=0$, and part (ii) of
Lemma \ref{le5b} turns out

\begin{lemma}
For $\|x\|=1$, let
$$
\Psi _{n,1}(x) =\frac{A_{0}^{2}}{\lambda }\sum_{k=0}^{n-[\log n]}
2^{2k}\,F_{k,m}(1) .
$$
Then
$$
\lim_{n\rightarrow \infty }\frac{\Psi _{n,1}(x) }{n^{2\mu +3}}=E_{0},
$$
where $E_{0}$ is given in (\ref{E0}) with $d=2$.

\end{lemma}

\bigskip

Moreover, Lemma \ref{le6} works for $d=2$, so finally Theorem \ref{th1} and Theorem \ref{th2} also hold
in this case.

\bigskip

\section*{Acknowledgements}

First, second, fourth and fifth authors (A.M.D., L.F., T.E.P., M.A.P.) thank MINECO of Spain and the European
Regional Development Fund (ERDF) through grants MTM2011--28952--C02--02 and
MTM2014--53171--P, and Jun\-ta de Andaluc\'{\i}a grant P11--FQM--7276 and research group FQM--384.

The work of the third author (D.S.L.) was supported in part by NSF grant DMS136208.

This work was finished during a visit of second author (L.F.) to School of Mathematics,
Georgia Institute of Technology.

\end{document}